\documentclass[10pt]{article}
\usepackage{amsmath}
\usepackage{amsfonts}
\usepackage{amssymb}
\usepackage[english]{babel}
\usepackage{graphicx}
\usepackage{amsthm}
\usepackage{hyperref}
\usepackage{accents}
\usepackage{empheq}
\usepackage [autostyle, english = american]{csquotes}
\MakeOuterQuote{"}
\numberwithin{equation}{section}
\theoremstyle{plain}
\newtheorem*{theorem*}{Theorem}
\newtheorem*{lemma*}{Lemma}
\newtheorem{theorem}{Theorem}
\newtheorem{lemma}{Lemma}[section]
\newtheorem{corollary}[lemma]{Corollary}

\newtheorem{proposition}[lemma]{Proposition}
\newtheorem{innercustomthm}{Theorem}
\newenvironment{customthm}[1]
{\renewcommand\theinnercustomthm{#1}\innercustomthm}
{\endinnercustomthm}

\theoremstyle{definition}
\newtheorem{definition}[lemma]{Definition}
\newtheorem{remark}[lemma]{Remark}

\def\pr{\partial}
  
\def\V{\Vert}

\usepackage{geometry}
\begin{document}
\title{Finite-time Singularity Formation for Strong Solutions to the Boussinesq System}
\author{Tarek M. Elgindi\textsuperscript{1} and In-Jee Jeong\textsuperscript{2}}
\footnotetext[1]{Department of Mathematics, UC San Diego. E-mail: telgindi@ucsd.edu.}
\footnotetext[2]{Department of Mathematics, Princeton University and UC San Diego. E-mail:
ijeong@princeton.edu}
\date{\today}
\maketitle

\begin{abstract}
As a follow up to our work \cite{EJSI}, we give examples of finite-energy and Lipschitz continuous velocity field and density $(u_0,\rho_0)$ which are $C^\infty$-smooth away from the origin and belong to a natural local well-posedness class for the Boussinesq equation whose corresponding local solution becomes singular in finite time. That is, while the sup norm of the gradient of the velocity field and the density remain finite on the time interval $t\in [0,1)$, both quantities become infinite as $t\rightarrow 1$. The key is to use scale-invariant solutions similar to those introduced in \cite{EJSI}. The proof consists of three parts: local well-posedness for the Boussinesq equation in critical spaces, the analysis of certain special infinite-energy solutions belonging to those critical spaces, and finally a cut-off argument to ensure finiteness of energy. All of this is done on spatial domains $\{(x_1,x_2):  x_1 \ge \gamma|x_2|\}$ for any $\gamma > 0$ so that we can get arbitrarily close to the half-space case. We show that the $2D$ Euler equation is globally well-posed in all of the situations we look at, so that the singularity is not coming from the domain or the lack of smoothness on the data but from the vorticity amplification due to the presence of a density gradient. It is conceivable that our methods can be adapted to produce finite-energy $C^\infty$ solutions on $\mathbb{R}^2_+$ which become singular in finite time. 
\end{abstract}
\tableofcontents

\section{Introduction}

The global regularity problem for the $3D$ Euler system is a well known open problem in mathematical fluid dynamics. While the Euler system was first derived over 250 years ago, still much is unknown about it and many other incompressible fluid models. Aside from the non-linearity of the models, they are also highly non-local: any disturbance in one portion of the fluid immediately affects the whole of the fluid. $3D$ Euler flows also enjoy very few known conserved quantities. The non-linearity, non-locality, and lack of conserved quantities, could lead one to believe that "anything" could happen to solutions to the $3D$ Euler system and similar models. Indeed, it is not even known if smooth solutions always remain smooth for all time though there is strong numerical evidence that smooth solutions may develop singularities in finite time, at least in some scenarios. Proving that solutions to the $3D$ Euler equation actually exhibit a given conjectured behavior is a formidable challenge with a lack of control on the dynamics of solutions due to non-linearity, non-locality, and the absence of conserved quantities. For this reason, if one wishes to prove something non-trivial about the dynamics of solutions to the $3D$ Euler system or a similar system, it is important to look at very specific classes of solutions where one may have better control on the solutions. We adopt this philosophy here to analyze the dynamics of a class of strong solutions to the $2D$ Boussinesq system and in a follow-up work we consider the $3D$ Euler system. 

An important piece in the global regularity puzzle was placed by Luo and Hou \cite{HouLuo} who have recently produced very strong numerical evidence that smooth solutions to the $3D$ axi-symmetric Euler system develop singularities in finite time when the fluid domain has a smooth solid boundary.  
Essentially, they showed that the presence of a solid boundary in the fluid domain stabilizes the solution enough to make the solution become singular in finite time. Thereafter, there were a number theoretical works which confirmed that singularity formation in the presence of a solid boundary holds on some $1D$ and $2D$ \emph{models} of the real $3D$ and $2D$ fluid systems. Here, we will work with the actual $2D$ Boussinesq system and we will show finite-time singularity formation for strong solutions when the fluid domain is a sector with angle less than $\pi$. 

Our main point of departure is that most fluid systems (such as the Boussinesq and $3D$ Euler systems) satisfy certain rotation, reflection, and/or scaling symmetries. This is to say that if the initial velocity field of the fluid satisfies these symmetries, the (unique) solution will continue to satisfy these symmetries. It is then conceivable that if one imposes enough symmetries on the initial data, one will be able to have good control on the respective solutions. Of course, these ideas are classical and are used throughout the analysis of PDE; however, the scaling symmetry in particular seems to never been effectively used in incompressible problems. This is mainly due to non-locality.  Here, we will use the scaling symmetry to build finite-energy strong solutions to the Boussinesq system which become singular in finite time. Roughly one can phrase the main results of this paper in the following terms. 
\begin{center} \emph{Let $\Omega_\gamma:=\{(x_1,x_2)\in\mathbb{R}^2: \gamma|x_2|\leq x_1\}.$ We give a space $X\subset W^{1,\infty}$ in which the Boussinesq system on $\Omega_\gamma$ can be solved uniquely for a length of time depending only on the size of the initial velocity field in $X$ for each $\gamma>0$. Moreover, we show that there are finite-energy solutions belonging to $X$ which leave $W^{1,\infty}$ after some finite amount of time. Finally, we show that all solutions in $X$ with zero density remain in $X$ for all time.}  \end{center}

This is to say that there are local strong solutions on the domains $\Omega_\gamma$ which become singular in finite time even though the same cannot happen for the $2D$ Euler equation. The presence of the boundary is important in our construction but we believe that it might be possible to smooth out the corner at $(0,0)$ and this will be taken up in future work. Independent of being able to smooth out the boundary, however, this result clearly sets apart the $2D$ Euler equation from the Boussinesq system and the $3D$ Euler equation since there is global regularity in the former and finite time blow-up in the latter cases.

\subsection{The $2D$ Boussinesq system}
Recall the $2D$ Boussinesq system which models the dynamics of an inviscid buoyant fluid\footnote{We are considering a slightly unconventional scenario where gravity is pushing horizontally and to the left for some notational convenience.}:
\begin{empheq}[left = \empheqlbrace]{align}
\label{B1} &\partial_t u + u\cdot\nabla u+\nabla p=  \left( \begin{array}{cc}
-\rho \\
0 \\
\end{array} \right),
\\
\label{B2} &\partial_t \rho + u\cdot\nabla \rho=0.
\end{empheq}
Here, $u$ is the velocity field of a two dimensional fluid $u:\Omega\times \mathbb{R}\rightarrow \mathbb{R}^2$ and $\rho:{\Omega}\times \mathbb{R}\rightarrow \mathbb{R}$ is the density of the fluid. As is usual, one takes $u$ to satisfy the no-penetration boundary condition $u\cdot n=0$ on $\partial\Omega$.
Just like the case of the $3D$ Euler equations, it was unknown whether strong solutions can become singular in finite time. This is due to a gap between known conserved quantities and what is needed to propagate smoothness. To our knowledge, the only known coercive global-in-time a-priori estimates for the Boussinesq system are the following:
$$\V\rho(t)\V_{L^p}=\V\rho_0\V_{L^p}$$ and $$\V u(t)\V_{L^2} \leq \V u_0\V_{L^2} + t \V\rho_0\V_{L^2}$$for all $t>0$ and all $1\leq p\leq \infty$. 
However, like the $3D$ Euler equations, an a-priori bound on $\V \nabla u\V_{L^\infty}$ is what is needed to ensure global regularity.  
In fact, the global regularity problem for the Boussinesq system is even discussed in Yudovich's "eleven great problems of mathematical hydrodynamics" \cite{Y3}. Here, we will show the existence of finite-energy strong solutions which become singular in finite time. 
\subsection{Analogy with the $3D$ axi-symmetric Euler equations}
Upon passing to the vorticity formulation for this system, we see clearly the relation between the Boussinesq system and the axi-symmetric Euler equations:
\begin{align*}
& \frac{D\omega}{Dt}=\partial_{x_2} \rho &   &\frac{\tilde{D}}{Dt}\Big(\frac{\omega_\theta}{r}\Big)=-\frac{1}{r^4} \partial_{x_3}[(ru^\theta)^2]   \\
&\frac{D\rho}{Dt}=0      &     &\frac{\tilde D}{Dt}\Big(r u^\theta\Big)=0 \\
&\frac{D}{Dt} = \partial_t +u_1\partial_{x_1} +u_2\partial_{x_2}            &     &\frac{\tilde D}{Dt}=\partial_t + u^r\partial_r+ u^3 \partial_{x_3} \\
&u_1 = \partial_{x_2}\psi, \quad u_2 = -\partial_{x_1} \psi                         &     &u^r=\frac{\partial_{x_3}\tilde\psi}{r}, \quad u^3=-\frac{\partial_r\tilde\psi}{r} \\
& L\psi = \omega, \quad  L=\partial_{x_1}^2+\partial_{x_2}^2         &   & \tilde{L}\tilde\psi=\frac{\omega^\theta}{r}, \quad \tilde{L}=\frac{1}{r}\partial_r (\frac{1}{r} \partial_r)+\frac{1}{r^2}\partial_{x_3}^2
\end{align*}
with the Boussinesq system (in vorticity form) being the system on the left and the axi-symmetric $3D$ Euler system on the right. 
According to some authors, the behavior of solutions to the Boussinesq system and the axi-symmetric $3D$ Euler equations away from the symmetry axis $r=0$ should be "identical" (\cite{MB}, \cite{EShu94}). In both models, there is a vorticity ($\omega$ or $\omega^\theta$) which produces a velocity field ($u$ or $(u^r,u_3)$) which advects a scalar quantity ($\rho$ or $r u^\theta$). Then a derivative of the advected quantity forces the vorticity. It is conceivable that in both of these situations, strong advection of the scalar quantity causes vorticity growth, and this vorticity growth causes stronger advection, and that uncontrollable non-linear growth occurs until singularity in finite time. Getting a hold of this mechanism requires strong geometric intuition and, seemingly, much more information than what is now known about the system. 
\subsection{Main result}
In \cite{EJSI}, we proved that a sufficient condition for blow-up of finite-energy strong solutions to the surface quasi-geoestrophic (SQG) equation is blow-up for scale-invariant (radially homogeneous) solutions to the SQG equation. These scale-invariant solutions satisfy a $1D$ equation, and blow-up for that $1D$ equation is still open though some progress has been made in our work \cite{EJDG}. Here, we extend the results of \cite{EJSI} to the Boussinesq system and also prove singularity formation for the associated $1D$ equation. In particular, the program introduced in \cite{EJSI} as applied to the Boussinesq system is completed here.  
We now state the main theorems. To do so, we must give a few definitions. First, let $\Omega$ be the spatial domain (whose boundary is depicted in Figure \ref{fig:setup} as thickened lines) $$\Omega:=\{(x_1,x_2)\in \mathbb{R}^2: |x_2| < x_1 \}.$$ Second, we define the scale of spaces $\mathring{C}^{0,\alpha}$ introduced in \cite{EJSI} and \cite{E1} using the following norm: 
$$\V f\V_{\mathring{C}^{0,\alpha}(\overline{\Omega})}:= \V f\V_{L^\infty(\overline{\Omega})} + \V |\cdot|^{\alpha} f\V_{{C}_*^\alpha(\overline{\Omega})}.$$ We recall some of the properties of this space in Section \ref{sec:prelim}.  This scale of spaces can be used to propagate boundedness of the vorticity, the full gradient of the velocity field, $\nabla u,$ as well as angular derivatives thereof. 
Next we will state the main theorems, which are local well-posedness for the Boussinesq system in $\mathring{C}^{0,\alpha}(\overline{\Omega})$ and finite-time singularity formation in the same space. 
\begin{figure}
\includegraphics[scale=0.8]{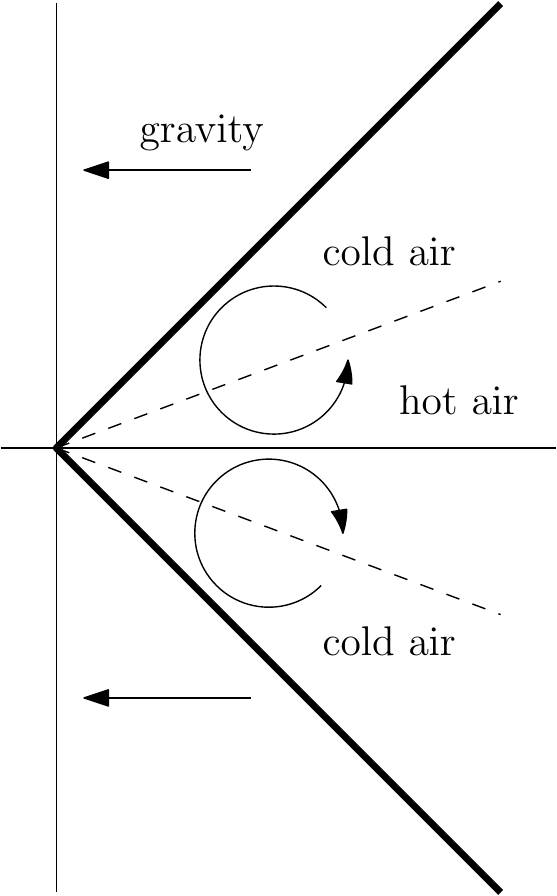} 
\centering
\caption{Setup for the domain and the data}
\label{fig:setup}
\end{figure}
\begin{customthm}{A}[Local well-posedness]\label{MainThm1}
Let $0<\alpha<1.$ Suppose that $\nabla\rho_0, \omega_0\in \mathring{C}^{0,\alpha}(\overline{\Omega})$ are given, where $\omega_0$ and $\partial_{x_2}\rho_0$ are odd and $\partial_{x_1}\rho_0$ is even with respect to the $x_1$-axis. Then, there exists a time $T=T(\V\nabla\rho_0\V_{\mathring{C}^{0,\alpha}},\V\omega_0\V_{\mathring{C}^{0,\alpha}})>0$ and a unique solution $(\omega,\rho)$ to the $2D$ Boussinesq system with $$\omega,\nabla\rho\in  C([0,T); \mathring{C}^{0,\alpha}(\overline{\Omega}))$$ and $(\omega,\rho)|_{t=0}=(\omega_0,\rho_0),$ satisfying the same set of symmetries. Moreover, the local solution $(\omega,\nabla\rho)$ cannot be continued past some $T^{*}<\infty$ if and only if $$\limsup_{t\rightarrow T^{*}} \int_0^t \V\nabla u(s)\V_{L^\infty}ds =+\infty\quad or\quad  \limsup_{t\rightarrow T^{*}}  \int_0^t \V\nabla \rho(s)\V_{L^\infty}ds=+\infty.$$
\end{customthm}
\begin{customthm}{B}[Global well-posedness when $\rho\equiv 0$]\label{MainThm2}
When $\rho_0\equiv 0$, the local solution of Theorem \ref{MainThm1} is global and $\V\omega\V_{\mathring{C}^{0,\alpha}}$ satisfies a double-exponential upper bound: \begin{equation}\label{eq:expexp}
\begin{split}
\V \omega(t)\V_{\mathring{C}^{0,\alpha}(\overline{\Omega})} \le C\exp(C\exp(Ct)),
\end{split}
\end{equation} for some $C > 0$ depending only on $\V \omega_0\V_{\mathring{C}^{0,\alpha}}$.
\end{customthm}

\begin{customthm}{C}[Singularity formation for compactly supported data]\label{MainThm3}
There exists a pair $\omega_0,\nabla\rho_0\in \mathring{C}^{0,\alpha}(\overline{\Omega})$ which are compactly supported in $\bar\Omega$ with the corresponding $u_0$ of finite energy such that the local solution of Theorem \ref{MainThm1} satisfies $$\limsup_{t\rightarrow T^*} \int_0^t \V\nabla u(s)\V_{L^\infty}ds=+\infty$$ for some $T^*<\infty$. 
\end{customthm}
\begin{customthm}{D}[Singularity formation for smooth data]\label{MainThm4}
There exists a pair $\omega_0,\nabla\rho_0\in L^\infty\cap {C}^{\infty}(\overline{\Omega})$ such that the local unique solution of Theorem \ref{MainThm1} satisfies $$\limsup_{t\rightarrow T^*} \int_0^t \V\nabla u(s)\V_{L^\infty}ds=+\infty$$ for some $T^*<\infty$. 
\end{customthm}
From Theorem \ref{MainThm3}, we have the following
\begin{corollary}\label{GeneralizedBlowUpProblem}
There exists a solution pair to the $2D$ Boussinesq system \eqref{B1} -- \eqref{B2}, $(u,\rho)\in W^{1,\infty}([0,1)\times\overline\Omega)$ of finite energy such that $\limsup_{t\rightarrow 1} \V\nabla u(t)\V_{L^\infty}=+\infty$. 
\end{corollary}
\begin{remark}
Theorem \ref{MainThm2} was established in our previous work \cite{EJSI}, when the domain is the entire plane, and vorticity is $m$-fold symmetric for some $m \ge 3$. This proof carries over to the current setup without much difficulty, with the logarithmic bound given in \eqref{eq:log}. 
\end{remark}
\begin{remark}
The strategy of the proof of Theorem \ref{MainThm3} is to first prove singularity formation for a certain class of infinite-energy solutions, and then show that the blow-up is stable with respect to certain kinds of perturbations (in particular, multiplication by a smooth cut-off function). The infinite energy solutions we construct are certainly not the first class of infinite-energy solutions which can be shown to be come singular in finite time \cite{SaWu}. However, ours seem to be the only known example which are stable under multiplication by a cut-off and which have bounded vorticity before the blow-up time. This seems to be essential to put the solutions into a natural uniqueness class. 
\end{remark}
\begin{remark}
We now make a few remarks regarding extensions of these results. 
\begin{itemize}
\item We note that the scale invariant solutions used to construct those solutions become singular on the whole line $x_1=x_2$, whereas the solutions are smooth away from $0$ before the blow-up time. It is possible that a more careful localization procedure can allow us to localize those scale-invariant solutions to $C^\infty$ solutions on $\mathbb{R}^2_+$ which become singular in finite time. 
\item It is very likely that a similar result holds for the $3D$ Euler equations and this is the subject of our forthcoming work. 
\item Theorems \ref{MainThm1} -- \ref{MainThm4} apply to all the spatial domains $\{(x_1,x_2)\in\mathbb{R}^2: \gamma|x_2| \leq x_1\}$ for $0 < \gamma <\infty$, with minor modifications (see Section \ref{sec:LWP}). In the case $\gamma > 1$, the extra assumptions that $\omega_0$ and $\partial_{x_2}\rho_0$ are odd and $\partial_{x_1}\rho_0$ is even in the local well-posedness result can be dropped, and moreover, $\omega$ and $\nabla\rho$ can actually be taken to be $C^{\alpha}$-regular and compactly supported (see Section \ref{sec:acutecase}). 
\item The data and solutions in Theorems \ref{MainThm1} -- \ref{MainThm4} can be taken to be as smooth as we like in the angular variable. 
\end{itemize}
\end{remark}

\subsection{Some ingredients of the proofs}

\subsubsection{Scale invariant solutions}
Let us recall that the Boussinesq system satisfies the following simple scaling property: whenever $(u(t,\cdot),\rho(t,\cdot))$ is a solution to the Boussinesq system \eqref{B1} -- \eqref{B2}, then $\frac{1}{\lambda} (u(t,\lambda\cdot),\rho(t,\lambda\cdot))$ is also a solution for any $\lambda>0$. This means that if $(u,\rho)$ can be placed in a natural existence/uniqueness class and if $\frac{1}{\lambda}(u_0(\lambda\cdot),\rho_0(\lambda\cdot))=(u_0,\rho_0)$ for all $\lambda>0$, then this property will be propagated so that $$\frac{1}{\lambda} (u(t,\lambda x),\rho(t,\lambda x))=(u(t,x),\rho(t,x))$$ for all $x\in\mathbb{R}^2$ and $\lambda,t>0.$ In particular, at a formal level, 1-homogeneity of $u$ and $\rho$ is propagated for all time. This means that such $u$ and $\rho$ will satisfy a $1D$ equation which will be easier to analyze than the full $2D$ system. We will prove here that this can be made rigorous and that blow-up for some finite-energy strong solutions to the Boussinesq system can be discovered this way. 
This is not the first work where scale-invariance is used in this way. In the context of the Navier-Stokes equation, the study of solutions arising from scale-invariant data is quite classical and goes back at least to Leray \cite{Leray34am}. In fact, Leray conjectured that such solutions could play a key role in the global regularity problem for the Navier-Stokes equation \cite{Leray34am}. It was later shown that self-similar blow-up for the Navier-Stokes equation is impossible under some very mild decay conditions in the important works \cite{NRS96} and \cite{TsaiSS98}. Of note also is that scale-invariant solutions to the $3D$ Navier-Stokes equation have been used to give sufficient conditions for non-uniqueness of Leray-Hopf weak solutions \cite{JiaSverak} (see also \cite{CP96}, \cite{JiaSverakSS}, \cite{TsaiFSS14}, and \cite{BradshawTsaiFSS17}). 
For fluid equations \emph{without} viscosity, it seems that the work \cite{EJSI} is the first where scale-invariant solutions were introduced, placed into natural local well-posedness classes, and analyzed. It is unclear why considering such solutions is not well-known in the community, even 250 years after Euler introduced his system! However, what we will show here is that they play a crucial role in the global regularity problem. They are the key to Theorem \ref{MainThm3}. 
\subsubsection{Monotonicity of the vorticity}
One of the further advantages of studying scale-invariant solutions is that they have been shown to propagate useful structures such as positivity and monotonicity \cite{EJSI}. Let us recall the equation for scale-invariant solutions to the $2D$ Euler equations:
\begin{empheq}[left = \empheqlbrace]{align}
\label{2dESI1} &\partial_t g +2G\partial_\theta g=0,\\
\label{2dESI2} &4G +\partial_{\theta\theta} G=g,
\end{empheq}
where we search for solutions which are $2\pi/m
$-periodic with $m\geq 3$ in order to be able to invert $\partial_{\theta\theta}+4$ in equation \eqref{2dESI2}. Here, the vorticity of the $2D$ fluid is just $\omega(t,r,\theta):=g(t,\theta)$. From the structure of \eqref{2dESI1}, it is clear that both a sign on $g$ and a sign\footnote{Note that $\partial_\theta g$ can only have a sign if we put solid boundaries.} on $\partial_\theta g$ can be propagated in time. Note that no such monotonicity is known on the vorticity for general $2D$ vorticities. Here we see that positivity of $\partial_\theta \omega$ can be propagated when $\omega$ is scale invariant. It would be remarkable if this could be extended to general $2D$ flows in a suitable sense. Having this type of monotonicity on the $1D$ solutions plays an important role in this work.  

\subsubsection{Scheme of the proof}

The proof of Theorems \ref{MainThm1} -- \ref{MainThm4} relies on several important observations. First, in view of the results on ill-posedness for the Boussinesq system at critical regularity \cite{EM1}, we must first find a small enough space to prove local well-posedness but a large enough one to accommodate $1-$homogeneous velocity fields. This is achieved through considering a certain class weighted H\"older spaces and proving sharp elliptic estimates on these spaces. It is important to remark that without any symmetry in the domain, such estimates are not possible. To achieve these estimates, we have to carefully study the properties of the Green's function on sectors  which is constructed using conformal mapping. The important point is that the Green's function on a sector has much better decay properties that the usual Newtonian potential since the conformal map of the sector onto the whole space sends $z \mapsto z^{\alpha}$ and $\alpha>2.$ This observation is essential in the local well-posedness argument. The second important aspect of the proof is the analysis of solutions with $1-$homogeneous velocity and density. Such solutions satisfy the $1+1$ dimensional system for two scalar quantities $g$ and $P.$ First one must prove that $g$ and $P$ satisfy certain symmetry and monotonicity properties so long as they exist. It is not at all obvious how these symmetry properties could be derived for the original $2D$ system. Then, one can show that a certain combination of quantities satisfies a Ricatti-type ODE which becomes singular in finite time. The third crucial step in the proof is the cut-off argument. To now, we have proven blow-up for $1-$homogeneous solutions. We should remark that this is interesting in itself since such solutions belong to a natural existence and uniqueness class for the equation---which already separates them from previous infinite energy blow-up proofs. However, it is desirable that the blow-up be for finite energy solutions with bounded density since these are conserved quantities for the system. This leads one to consider data which is locally $1-$homogeneous near $0$ but compactly supported or rapidly decaying away at infinity. This leads to a considerable error which must be controlled. How this error is controlled is based on two observations. The first is that the velocity field of a $C^\alpha$ vorticity which vanishes near $0$ is vanishing of order $|x|^{1+\alpha}$. This is again due to the rapid decay of the Green's function on a sector.  The second observation is a simple product rule which is contained in Lemma \ref{lem:Holder_product}.

\subsection{Previous results}
The Boussinesq system seems to first have been derived by Rayleigh in 1916 \cite{Rayleigh1916} to model the motion of a buoyant fluid and to explain what came to be known as Rayleigh-B\'enard convection. Since then, the system has been used in many different contexts such as Rayleigh-B\'enard convection, atmospheric dynamics, oceanic dynamics, and as a model of the $3D$ Euler equations (see the books \cite{Majda03} and \cite{DoeringGibbon}). For this reason, there are many works devoted to the Boussinesq system and we will only cite a few belonging to three general categories all related to the blow-up problem: analytical works, numerical works, and works on simplified models of the Boussinesq system. 
\subsubsection*{Analytical works on the inviscid Boussinesq system}
Most of the works on the inviscid Boussinesq system are local-in-time results for suitably regular solutions. Local well-posedness for smooth solutions and various blow-up criteria in subcritical Besov and Sobolev spaces have been established by a number of authors (see \cite{EShu94}, \cite{ChaeNam97}, \cite{BCD}, and \cite{Danchin2013}). Though, it appears to be an open problem to determine whether the classical Beale-Kato-Majda criterion even applies to the Boussinesq system (see \cite{WuLectureNotes} and \cite{HORY2016}). Local well-posedness for Yudovich type data or vortex patch type data has been investigated in various works (\cite{HassainiaHmidi}, \cite{DanchinPaicu}) though the \emph{inviscid} Boussinesq system has actually been shown to be ill-posed in the Yudovich class \cite{EM1}. There have also been numerous works on the $2D$ Boussinesq system with different kinds of dissipative mechanisms which are known to model different physical scenarios. For example, there are works with the inclusion of viscosity or partial viscosity into either or both equations of the Boussinesq system (\cite{Chae2006}, \cite{DanchinPaicu2011}, \cite{CaoWu2013}, \cite{LiTiti16}) and the inclusion of fractional dissipation in either or both equations (\cite{HKR1}, \cite{HKR2}, \cite{CV}). We also mention the work of Chae, Constantin, and Wu \cite{CCW} where it is proven that the gradient of the density profile for solutions to the Boussinesq equation may grow exponentially in time in the same type of domain that we consider. 
\subsubsection*{Models of the Boussinesq system}
In the last few years, particularly after the important works of Hou and Luo \cite{HL} and Kiselev and Sverak \cite{KS}, the idea of using the boundary to better control singularity formation blossomed. Thereafter, a number of model equations have been put forth to model both the axi-symmetric $3D$ Euler equations and the Boussinesq system. Some of these models are one-dimensional and are based on asymptotic expansions of the velocity, vorticity, and/or density near the boundary. Examples are the models of Hou and Luo \cite{HLModel} and Choi, Kiselev, and Yao \cite{CKY}. Unfortunately, there does not seem to be any clear way to pass from results for the $1D$ models to the full $2D$ models. This is an important point which differentiates this work: via using scaling-invariance it is possible to pass from $1D$ results to $2D$ or $3D$ results. Another class of models follow, in spirit, the model introduced by Constantin, Lax, and Majda \cite{CLM}. Finally, we mention that T. Tao has devised some different models of the $3D$ Euler equations which blow-up in finite time \cite{TaoEuler}. These models share some structural similarities with the $3D$ Euler equations -- particularly the conservation of energy -- and they indicate that the obvious conservation laws are not enough to ensure global regularity.  
\subsubsection*{$2D$ Euler on domains with acute corners}
Since Theorems \ref{MainThm1} -- \ref{MainThm4} are set on a domain with a single corner, we discuss some of the relevant works on the $2D$ Euler equations on domains with corner. We do this to emphasize that the singularity formation we prove in Theorems \ref{MainThm3} and \ref{MainThm4} is neither coming from the regularity of the data nor from the corner but from a genuine non-linear cascade in the Boussinesq system. 
Recall that for domains with smooth boundary, Yudovich has shown (\cite{Y1}) that for given $\omega_0 \in L^1\cap L^\infty$, there is a unique global solution to the $2D$ Euler equations with bounded vorticity. Yudovich's theorem have been successfully extended to domains with corners. We mention the works of Bardos-Di Plinio-Temam \cite{BDT}, Lacave-Miot-Wang \cite{LMW},  Lacave \cite{L}, and Di Plinio-Temam \cite{DT}. The references \cite{LMW} and \cite{L} are based on straightening out the domain using a bi-holomorphism and studying the behavior of the Biot-Savart law under this transformation. The works \cite{BDT} and \cite{DT} seem to be inspired by Grisvard's important work on elliptic problems in singular domains \cite{Gris} and deal directly with estimates for certain kinds of singular integrals to extend the Yudovich theory. These works show that the Yudovich theory extends completely to polygonal domains whenever the angles are less than or equal to $\pi/2$ and this is exactly the setting we are in. Likewise, in the case of acute angles, the $C^{k,\alpha}$-theory in smooth domains extends to polygonal domains (this is a by-product of the well-posedness theory in Section \ref{sec:LWP} -- there is a restriction on the range of $k,\alpha$ depending on the angle). In more exotic domains, singularity formation for the $2D$ Euler equations is possible \cite{KZ} but we do not consider such domains here since they seem to have little bearing on the actual dynamics of strong solutions to the Euler equations on more regular domains.  We also remark that while we do consider corner domains with obtuse angles as well, we impose an odd reflection symmetry on the vorticity in that case which precludes any problems at the level of $2D$ Euler as well. It is likely that if one were to consider the $2D$ Euler equation on a domain with an \emph{obtuse} corner but without an odd symmetry then there is a type of ill-posedness or "finite-time singularity" which could occur but this is unrelated to the present work.  

\subsubsection*{Numerical Works}
We close this part by mentioning a few of the numerous numerical works on the Boussinesq system and the axi-symmetric $3D$ Euler system. Work of Pumir and Siggia seem to indicate singularity formation for the $3D$ axi-symmetric Euler equations \cite{PumirSiggia}. For the Boussinesq system is that of E and Shu \cite{EShu94} where no singularity formation is observed numerically for initial data similar to that of Pumir and Siggia. We also refer the reader to the survey paper of Gibbon \cite{Gibbon} and the recent important work of Hou and Luo \cite{HouLuo} for a more in-depth discussion.

\subsection{Organization of the paper} 
The rest of this paper is organized as follows. In Section \ref{sec:prelim}, we define precisely the scale invariant H\"{o}lder spaces, and prove a few simple properties regarding functions belonging to such spaces.  Then, in Section \ref{sec:LWP}, we prove our first main result, Theorem \ref{MainThm1}, which is the local well-posedness result for $2D$ Boussinesq in the scale of $\mathring{C}^{k,\alpha}$-spaces. For this purpose, it is essential to have sharp H\"older estimates on domains with corners, and this issue is discussed in detail in Subsection \ref{subsec:gris}. The proof of local well-posedness for the $1D$ system is given as well. After that, in Section \ref{sec:1D}, we analyze the $1D$ system in some detail, and prove in particular that there is finite time blow-up for a wide class of smooth initial data. In Sections \ref{sec:compact_blowup} and \ref{sec:smooth_blowup}, we prove Theorems \ref{MainThm3} and \ref{MainThm4}, respectively. These results are based on the blow-up for the $1D$ system and yet another local well-posedness result proved in Section \ref{sec:compact_blowup} (Theorem \ref{thm:1Dplus2D}), which says that a solution to the $2D$ Boussinesq system must blow-up if its ``scale-invariant part'' blows up in finite time. Finally, in Section \ref{sec:acutecase}, we prove that when the angle is acute, the blow-up happens for uniformly (up to the corner) H\"older continuous and compactly supported vorticity. 
\subsection*{Notations}
As it is usual, we use letters $C, c,\cdots$ to denote various positive absolute constants whose values may vary from a line to another. We write functions depending on time and space as $f(t,\cdot) = f_t(\cdot)$, and partial derivatives in time and space are respectively denoted by $\partial_t f$ and $\partial_{x_i} f$, where $i = 1, 2$. We shall often use the polar coordinates system, with the usual convention that $x_1 = r\cos\theta$ and $x_2 = r\sin\theta$. Partial derivatives with respect to $r$ and $\theta$ are denoted by $\pr_r$ and $\pr_{\theta}$, respectively. For a scalar function $f$ defined on a subset of $\mathbb{R}^2$, we use the notation $\nabla f = (\partial_{x_1}f, \partial_{x_2}f)^T$ as well as $\nabla^\perp f = (-\partial_{x_2}f , \partial_{x_1}f)^T$. The symbol $\perp$ represents the counter-clockwise rotation by $\pi/2$ in the plane. 

\section{Preliminaries}\label{sec:prelim}
In this section, let $D$ be some subset of the plane. Then, the scale-invariant H\"{o}lder spaces are defined as follows:
\begin{definition}
Let $0 < \alpha \le 1$. Given a function $f \in C^0(D\backslash\{0\})$, we define the $\mathring{C}^{0,\alpha}(\overline{D}) = \mathring{C}^{\alpha}(\overline{D})$-norm by \begin{equation*}
\begin{split}
\V f \V_{\mathring{C}^{\alpha}(\overline{D})} &:= \V f \V_{L^\infty(\overline{D})} + \V |\cdot|^{\alpha}f \V_{{C}_*^\alpha(\overline{D})} \\
& := \sup_{x \in\overline{ D}} |f(x)| + \sup_{x,x' \in \overline{D}, x \ne x'} \frac{||x|^\alpha f(x)- |x'|^\alpha f(x')|}{|x-x'|^\alpha}.
\end{split}
\end{equation*} Then, for $k \ge 1$, we define $\mathring{C}^{k,\alpha}$-norms for $f \in C^k(D\backslash \{0\})$ by \begin{equation*}
\begin{split}
\V f \V_{ \mathring{C}^{k,\alpha}(\overline{D})} := \V f \V_{ \mathring{C}^{k-1,1}(\overline{D})} + \V |\cdot|^{k+\alpha} \nabla^k f \V_{{C}_*^\alpha(\overline{D})}.
\end{split}
\end{equation*} Here, $\nabla^k f$ is a vector consisting of all expressions of the form $\pr_{x_{i_1}} \cdots \pr_{x_{i_k}} f$ where $i_1,\cdots i_k \in \{ 1, 2\}$. Finally, we may define the space $\mathring{C}^\infty$ as the set of functions belonging to all $\mathring{C}^{k,\alpha}$: \begin{equation*}
\begin{split}
\mathring{C}^\infty := \cap_{k \ge 0, 0 < \alpha \le 1} \mathring{C}^{k,\alpha}.
\end{split}
\end{equation*}
\end{definition}
\begin{remark} From the definition, we note that:
\begin{itemize}
\item Let $D = \{ (r,\theta) : r > 0, \theta_1 < \theta < \theta_2 \}$. For a radially homogeneous function $f$ of degree zero, that is, $f(r,\theta) = \tilde{f}(\theta)$ for some function $\tilde{f}$ defined on $[\theta_1,\theta_2]$, we have \begin{equation*}
\begin{split}
\V f \V_{ \mathring{C}^{k,\alpha}(\overline{D})} = \V \tilde{f} \V_{C^{k,\alpha}[\theta_1,\theta_2]}. 
\end{split}
\end{equation*} Similarly, $f \in \mathring{C}^\infty(\overline{D})$ if and only if $\tilde{f} \in C^\infty[\theta_1,\theta_2]$.
\item If $f$ is bounded, then $\V |\cdot|^\alpha f\V_{{C}_*^\alpha} < + \infty$ if and only if (assuming that $|x'| \le |x|$) \begin{equation*}
\begin{split}
\sup_{x \ne x', |x-x'|\le c|x|} \frac{ |x|^\alpha  }{|x-x'|^\alpha}|f(x) - f(x')| < + \infty 
\end{split}
\end{equation*} for some $c > 0$. 
\end{itemize}
\end{remark}
\begin{lemma}[Product rule]\label{lem:Holder_product}
Let $f \in C^\alpha$ with $f(0) = 0$ and $h \in \mathring{C}^\alpha$. Then, we have the following product rule: \begin{equation}\label{eq:Holder_product}
\begin{split}
\V fh\V_{C^\alpha} \le C\V h\V_{\mathring{C}^\alpha} \V f \V_{C^\alpha}.
\end{split}
\end{equation}
\end{lemma}
\begin{proof}
Clearly we have that $\V fh\V_{L^\infty} \le \V f \V_{L^\infty} \V h \V_{L^\infty}$. Then, take two points $x \ne x'$ and note that \begin{equation*}
\begin{split}
\frac{f(x)h(x) - f(x')h(x')}{|x-x'|^\alpha} = h(x)\frac{f(x) - f(x')}{|x-x'|^\alpha} + \frac{f(x')}{|x'|^\alpha} \cdot \left( \frac{|x|^\alpha h(x) - |x'|^\alpha h(x')}{|x-x'|^\alpha} + \frac{|x'|^\alpha - |x|^\alpha}{|x-x'|^\alpha} h(x) \right)
\end{split}
\end{equation*} holds. The desired bound follows immediately. 
\end{proof}
\begin{remark}
Note that $C^\alpha\not\subset \mathring{C}^\alpha$ since functions belonging to $\mathring{C}^\alpha$ must, in a sense, have decaying derivatives. For example, a function $f\in \mathring{C}^{0,1}$ if and only if it is uniformly bounded and satisfies $|\nabla f(x)|\lesssim \frac{1}{|x|}$ almost everywhere. Of course, any compactly supported $C^\alpha$ function belongs to $\mathring{C}^\alpha$. 
\end{remark}
\section{Local well-posedness results}\label{sec:LWP}
In this section, we prove that the $2D$ Boussinesq system \eqref{B1} -- \eqref{B2} is locally well-posed in the scale of spaces $\mathring{C}^{\alpha}(\overline{\Omega})$. The key is to have sharp $C^\alpha$ and $\mathring{C}^\alpha$-estimates on sectors for the singular integral operators $\nabla^2(-\Delta)^{-1}$. We further prove a borderline inequality similar to the Kato inequality used in \cite{BKM} to deduce that the $2D$ Euler system is globally well-posed for odd vorticity in $\mathring{C}^\alpha(\bar\Omega).$  
\subsection{H\"{o}lder estimates for sectors}\label{subsec:gris}

We consider functions defined on the sectors \begin{equation*}
\begin{split}
\Omega_\beta := \{ (r,\theta) : 0 < \theta < \beta\pi \}
\end{split}
\end{equation*} for $0 < \beta < 1/2$. The positive quadrant corresponds to the case $\beta = 1/2$. Since we are interested in solving $\Delta \Psi = f$ on $\Omega_\beta$ with Dirichlet boundary conditions, we may as well view $f$ as defined on $\{ (r,\theta) : -\beta\pi < \theta < \beta\pi \}$ as an odd function with respect to the axis $x_2 = 0$. 

We shall take advantage of the explicit form for the Green's function on sectors:
\begin{lemma}\label{lem:Green}
	The Dirichlet Green's function on $\Omega_\beta$ is given by \begin{equation*}
	\begin{split}
	G_\beta(z,w) = \frac{1}{2\pi} \ln\left( \frac{|z^{1/\beta}-w^{1/\beta}|}{|\overline{z^{1/\beta}}-w^{1/\beta}|} \right).
	\end{split}
	\end{equation*} Similarly, the Dirichlet Green's function on  $\Omega_\beta \cap \{ r <R \}$ takes the form \begin{equation*}
	\begin{split}
	G_\beta^R(z,w) = \frac{1}{2\pi}  \ln\left( \frac{|z^{1/\beta} - w^{1/\beta}|}{|\overline{z^{1/\beta}} - w^{1/\beta}|} \cdot \frac{|R^2 - z^{1/\beta}w^{1/\beta}|}{|R^2 - \overline{z^{1/\beta}}w^{1/\beta}|} \right) .
	\end{split}
	\end{equation*}
\end{lemma}
In the above lemma, we are using the convention $z^\gamma := r^\gamma e^{i\theta\gamma}$ for $z = re^{i\theta}$ with $-\pi/2 < \theta < \pi/2$. The bar denotes the complex conjugate. Differentiating the Green's kernel in $z$, we obtain \begin{equation*}
\begin{split}
K_\beta(z,w) := \pr_z G_\beta(z,w) = - \frac{z^{1/\beta -1}}{4\pi\beta} \cdot \frac{\overline{w^{1/\beta}} - w^{1/\beta}}{(z^{1/\beta} - \overline{w^{1/\beta}})(z^{1/\beta} - w^{1/\beta})},
\end{split}
\end{equation*} and note that the partial derivatives $\partial_{x_1}\Psi$ and $\partial_{x_2}\Psi$ are given by the real and imaginary parts of $K_\beta * f$, respectively. 

The following lemma asserts that there is always a unique solution to the Poisson problem for the sector $\Omega_\beta$, even for functions not necessarily decaying at infinity. This is a generalization of the uniqueness statement given in \cite[Lemma 2.6]{EJSI} for sectors of angle $2\pi/m$ where $m$ is an integer larger than $2$.
\begin{lemma}[Existence and uniqueness for the Poisson problem]\label{lem:Uniqueness_Poisson}
	Let $f$ be a bounded function in the sector $\Omega_\beta$. Then there exists a unique solution to \begin{equation*}
\begin{cases}
		 \Delta\Psi  = f  \qquad&\mbox{in}\qquad \Omega_\beta,\\
	 \Psi = 0 \qquad&\mbox{on}\qquad \partial\Omega_\beta,
	\end{cases}
	\end{equation*} with $\Psi \in W^{2,p}_{loc}$ for all $p < \infty$ and $|\Psi(x)| \le C|x|^2 \ln(e + |x|)$. 
\end{lemma}
\begin{proof}
	For the existence, we simply define \begin{equation}\label{eq:Green_integral}
	\begin{split}
	\Psi(z) = \lim_{R \rightarrow +\infty} \int_{\Omega_\beta \cap \{ |w| < R \} } G_\beta(z,w) f(w) dw,
	\end{split}
	\end{equation} using complex notation. Using \begin{equation*}
	\begin{split}
	|G_\beta(z,w)| = \frac{1}{2\pi} \ln\left| 1 + \frac{z^{1/\beta} - \overline{z^{1/\beta}}}{\overline{z^{1/\beta}} - w^{1/\beta}} \right| \le \frac{1}{2\pi}  \frac{|z^{1/\beta} - \overline{z^{1/\beta}}|}{|\overline{z^{1/\beta}} - w^{1/\beta}|},
	\end{split}
	\end{equation*} one sees that the kernel decays as $|w|^{1/\beta}$ for each fixed $z$, and since $1/\beta >2 $, the limit is well-defined. We claim that $\Psi$ satisfies $|\Psi(z)| \le C|z|^2$. To see this, we differentiate the expression for $\Psi(z)$ to obtain \begin{equation*}
	\begin{split}
	\frac{|\nabla\Psi(z)|}{|z|} &\le \frac{1}{4\pi\beta}\int_{\Omega_\beta} |z|^{1/\beta - 2} \frac{|\overline{w^{1/\beta}} - w^{1/\beta}|}{|z^{1/\beta} - \overline{w^{1/\beta}}| |z^{1/\beta} - w^{1/\beta}|} |f(w)|dw \\
	&\le \frac{1}{4\pi\beta}\V f \V_{L^\infty} \int_{\Omega_\beta} \frac{|\overline{\zeta^{1/\beta}} - \zeta^{1/\beta}|}{|(\frac{z}{|z|})^{1/\beta} - \overline{\zeta^{1/\beta}}| | (\frac{z}{|z|})^{1/\beta} - \zeta^{1/\beta}| } d\zeta
	\end{split}
	\end{equation*} with a change of variables $w = |z|\zeta$. The last integral is bounded by a constant $C = C(\beta)$ uniformly in $z \in \Omega_\beta$, again because $1/\beta > 2$. Hence, $\Psi$ may grow at most quadratically as $|z| \rightarrow +\infty$. 
	
	Now we deal with the uniqueness statement. It suffices to show that when $f \equiv 0$, $\Psi \equiv 0$ is the only solution satisfying the assumptions on $\Psi$. Take some large $R > 0$, and as it is well-known, the values of $\Psi$ inside the ball $B_0(R)$ are determined by its trace on the boundary: \begin{equation*}
	\begin{split}
	\Psi(z) = \int_{\partial B_0(R)} \Re\left[ \pr_w G_\beta^R(z,w) \frac{\bar{w}}{R} \right] \Psi(w) d\sigma(w)
	\end{split}
	\end{equation*} where $d\sigma(w)$ represents the natural measure on the circle $\partial B_0(R)$. Explicit computations give that \begin{equation*}
	\begin{split}
	\pr_w G_\beta^R(z,w) &= - \frac{w^{1/\beta -1}(\overline{z^{1/\beta}} - z^{1/\beta})}{4\pi\beta} \\
	&\qquad \times  \left[ \frac{1}{(w^{1/\beta} - \overline{z^{1/\beta}})(w^{1/\beta} - z^{1/\beta})} - \frac{R^2}{(R^2 - w^{1/\beta}z^{1/\beta})(R^2 - w^{1/\beta}\overline{z^{1/\beta}})}  \right],
	\end{split}
	\end{equation*} and note that for each fixed $|z| < R$, the kernel satisfies the decay \begin{equation*}
	\begin{split}
	\left| \Re\left[ \pr_w G_\beta^R(z,w) \frac{\bar{w}}{R} \right]   \right| \le C \frac{1}{|w|^{1/\beta + 1}} = \frac{C}{R^{1/\beta + 1}}. 
	\end{split}
	\end{equation*} Then, \begin{equation*}
	\begin{split}
	|\Psi(z)| \le C R \frac{R^2 \ln( 1 + R)}{R^{1/\beta + 1}}  \rightarrow 0
	\end{split}
	\end{equation*} as $R \rightarrow +\infty$, since $1/\beta + 1 > 3$. 
\end{proof}

\begin{remark}
	In a recent work of Itoh, Miura, and Yoneda \cite{IMY2}, an expression of the Green's function of the form given in Lemma \ref{lem:Green} was used to prove the bound $|\nabla\Psi(x)|/|x| \lesssim \V f \V_{L^\infty}$, among other things. 
\end{remark}

\begin{remark}
	We note that when $f \equiv 1$, the function $$\Psi(x_1,x_2) := \frac{x_2^2 - \tan(\beta\pi)x_1x_2}{2}$$ vanishes on $\partial\Omega_\beta$ as well as $\Delta \Psi \equiv 1$. Therefore, by the above lemma, this $\Psi$ coincides with the integral representation given in \eqref{eq:Green_integral} with $f \equiv 1$. 
\end{remark}

\begin{lemma}[$\mathring{C}^\alpha$-estimate]\label{lem:C^circlealpha}
	Let $f \in \mathring{C}^\alpha(\overline{\Omega}_\beta)$, and $\Psi$ be the unique solution of $\Delta\Psi = f$ given in Lemma \ref{lem:Uniqueness_Poisson}. Then, we have the following bounds:   \begin{equation}\label{eq:log}
	\begin{split}
	\V \nabla^2\Psi\V_{L^\infty} \le C\V f\V_{L^\infty}\left(1 + \ln\left(1 + c \frac{\V f \V_{\mathring{C}^\alpha}}{\V f \V_{L^\infty}} \right) \right),
	\end{split}
	\end{equation} and 
	\begin{equation}\label{eq:C^circlealpha}
	\begin{split}
	\V\nabla^2\Psi\V_{\mathring{C}^\alpha(\overline{\Omega}_\beta)} \le C\V f \V_{\mathring{C}^\alpha(\overline{\Omega}_\beta)},
	\end{split}
	\end{equation}   where $C = C(\alpha,\beta) > 0$ is a constant depending only on $0 < \alpha <1$ and $0 < \beta < 1/2$. 
\end{lemma}
\begin{proof}
	We begin with the $L^\infty$ bound. We write the integral as \begin{equation*}
	\begin{split}
	\left[ \int_{ \Omega_\beta \cap \{ |z-w| \le l|z| \} } + \int_{ \Omega_\beta \cap \{ l|z| \le |z-w| \le 2|z|   \} }  +\int_{ \Omega_\beta \cap \{ 2|z| \le |z-w|   \} }    \right]P_\beta(z,w)(f(w) - f(z)) dw ,
	\end{split}
	\end{equation*} for some $l \le 1/2$ to be chosen later, and in the first region, we note that \begin{equation*}
	\begin{split}
	\left|\int_{ \Omega_\beta \cap \{ |z-w| \le l|z| \} } P_\beta(z,w)(f(w) - f(z)) dw \right| \le C \V f\V_{\mathring{C}^\alpha} \int_{|z-w| \le l|z|} \frac{|z-w|^{\alpha-2}}{|z|^\alpha}  dw \le Cl^\alpha \V f\V_{\mathring{C}^\alpha}.
	\end{split}
	\end{equation*} In the second region, we use the upper bound \eqref{eq:P_upper} to obtain \begin{equation*}
	\begin{split}
	\left| \int_{ \Omega_\beta \cap \{ l|z| \le |z-w| \le 2|z|   \} } P_\beta(z,w)(f(w) - f(z)) dw \right| \le C \ln\left(\frac{2}{l}\right) \V f \V_{L^\infty}. 
	\end{split}
	\end{equation*} Lastly, in the third region, note that $|w| \ge c|z|$ and therefore $|P_\beta(z,w)| \le C |z|^{1/\beta-2}/|w|^{1/\beta}$. Then, we can simply bound the integral by $C\V f \V_{L^\infty}$. Choosing \begin{equation*}
	\begin{split}
	l:= \min\left( \frac{1}{2}, \frac{\V f \V_{L^\infty}}{\V f \V_{\mathring{C}^\alpha}} \right)
	\end{split}
	\end{equation*} gives the logarithmic estimate \eqref{eq:log}.
	
	Now we need to obtain a bound on \begin{equation*}
	\begin{split}
	\frac{1}{|z-z'|^\alpha}\left( |z|^\alpha\int_{\Omega_\beta} P_\beta(z,w)(f(w) - f(z))dw - |z'|^\alpha\int_{\Omega_\beta} P_\beta(z',w)(f(w) - f(z'))dw   \right),
	\end{split}
	\end{equation*} but observe that in the presence of the $L^\infty$-estimate, it suffices to consider pairs $z \ne z' \in \Omega_\beta$ satisfying  $(1-1/10)|z| \le |z'| \le (1+1/10)|z|$, and also replace the factor $|z'|^\alpha$ by $|z|^\alpha$. Then we write the difference as:
	\begin{equation*}
	\begin{split}
	\frac{|z|^\alpha}{|z-z'|^\alpha} &\left[ \int_{\Omega_\beta\cap \{ |z-w| \le 2|z-z'| \} } P_\beta(z,w)(f(w) - f(z)) dw \right. \\
	&\qquad+ \int_{\Omega_\beta\cap \{ |z-w| \le 2|z-z'| \} } P_\beta(z',w)(f(w) - f(z')) dw  \\
	&\qquad+ \int_{\Omega_\beta\cap \{ |z-w| \ge 2|z-z'| \} } \left(P_\beta(z',w) - P_\beta(z,w)\right)(f(w) - f(z)) dw \\
	&\qquad+ \left.(f(z)-f(z'))\int_{\Omega_\beta\cap \{ |z-w| \ge 2|z-z'| \} } P_{\beta}(z',w)dw \right] =: I + II + III + IV.  
	\end{split}
	\end{equation*} It is straightforward to see that the term $IV$ is bounded by $C \V f \V_{\mathring{C}^\alpha}$, recalling from the previous lemma that the integral \begin{equation*}
	\begin{split}
	\left|\int_{\Omega_\beta\cap \{ |z-w| \ge 2|z-z'| \} } P_{\beta}(z',w)dw \right|
	\end{split}
	\end{equation*} is uniformly bounded. The terms $I$ and $II$ are simply bounded by using \eqref{eq:P_upper} and that \begin{equation*}
	\begin{split}
	|f(w) - f(z)| \le C \V f \V_{\mathring{C}^\alpha} \frac{|w-z|^\alpha}{|z|^\alpha}.
	\end{split}
	\end{equation*} Finally, to treat the term $III$, we further divide the domain of integration into $\{ |z|/2 \ge |z-w| \ge 2|z-z'| \}$ and $\{ |z-w| \ge |z|/2 \}$. Using the mean value theorem as in the proof of the previous lemma, we have in the first region \begin{equation*}
	\begin{split}
	&\frac{|z|^\alpha}{|z-z'|^\alpha}\left| \int_{\Omega_\beta\cap \{ |z|/2 \ge |z-w| \ge 2|z-z'| \}} \left( P_\beta(z',w) - P_\beta(z,w) \right)(f(w) - f(z)) dw  \right| \\
	&\qquad \le C \frac{|z|^\alpha}{|z-z'|^\alpha}\V f \V_{\mathring{C}^\alpha}\\
	&\qquad\qquad \times|z-z'| \int_{|z|/2 \ge |z-w| > 2|z-z'| } \left(  \frac{|z|^{1/\beta -3}}{|z-w|^2 (|z|^{1/\beta-2} + |w|^{1/\beta-2})} + \frac{1}{|z-w|^{3}} \right)  \frac{|z-w|^\alpha}{|z|^\alpha} dw \\
	&\qquad \le C \V f \V_{\mathring{C}^\alpha} |z - z'|^{1-\alpha} \left( \frac{|z|^\alpha - |z-z'|^\alpha}{|z|} + (|z-z'|^{\alpha -1 } - |z|^{\alpha -1}) \right) \le C  \V f \V_{\mathring{C}^\alpha} .
	\end{split}
	\end{equation*} Then, in the case $|z-w| \ge |z|/2$, we have $|w| \ge |z|/2$ and then simply use the $L^\infty$-bound for $f$ to deduce \begin{equation*}
	\begin{split}
	&\frac{|z|^\alpha}{|z-z'|^\alpha}\left| \int_{\Omega_\beta\cap \{ |z|/2 \le |z-w| \}} \left( P_\beta(z',w) - P_\beta(z,w) \right)(f(w) - f(z)) dw  \right| \\
	&\qquad \le C\V f \V_{L^\infty}\frac{|z|^\alpha}{|z-z'|^\alpha} \cdot \frac{|z-z'|}{|z|}\le C \V f\V_{L^\infty}.
	\end{split}
	\end{equation*} This finishes the proof. 

\end{proof} 

We now consider functions $f$ on $\Omega_\beta$ which admits a continuous extension up to the boundary, and is uniformly $C^\alpha$ on $\overline{\Omega}_\beta$: that is, \begin{equation*}
\begin{split}
\sup_{x \in \overline{\Omega}_\beta} |f(x)| + \sup_{x\ne y \in \overline{\Omega}_\beta} \frac{|f(x)-f(y)|}{|x-y|^\alpha} = \V f \V_{L^\infty(\Omega_\beta)} +  \V f \V_{{C}_*^\alpha(\overline{\Omega}_\beta)} =: \V f \V_{C^\alpha(\overline{\Omega}_\beta)} < + \infty. 
\end{split}
\end{equation*}
This result, while technically involved, is necessary to complete the cut-off argument in Section \ref{sec:compact_blowup}.
\begin{lemma}[$C^\alpha$-estimate]\label{lem:C^alpha}
	Let $f \in C^\alpha(\overline{\Omega}_\beta)$, and $\Psi$ be the unique solution of $\Delta\Psi = f$ given in Lemma \ref{lem:Uniqueness_Poisson}. Then, we have \begin{equation}\label{eq:L^infty_C^alpha}
	\begin{split}
	|\nabla^2\Psi(x)| \le C \ln(2 + |x|) \V f \V_{L^\infty(\Omega_\beta)} +  C\V f \V_{C^\alpha(\overline{\Omega}_\beta)}
	\end{split}
	\end{equation} and if in addition we have $0 < \alpha < 1/\beta - 2$, the following uniform H\"older estimate is valid: \begin{equation}\label{eq:dotCalpha}
	\begin{split}
	\V \nabla^2 \Psi \V_{{C}_*^\alpha(\overline{\Omega}_\beta)} \le C \V f \V_{C^\alpha(\overline{\Omega}_\beta)},
	\end{split}
	\end{equation} where $C = C(\alpha,\beta) > 0$ is a constant depending only on $0 < \alpha <1$ and $0 < \beta < 1/2$. 
\end{lemma}
\begin{remark}
We note that the logarithmic factor in \eqref{eq:L^infty_C^alpha} is only due to the fact that a general smooth function which is bounded in $C^\alpha$ function need not have decaying derivatives. Indeed, this logarithmic factor is not present if the functions are also bounded in $\mathring{C}^\alpha$. 
\end{remark}


\begin{proof}
	We note that to estimate $\nabla^2\Psi$ in $C^\alpha$, it suffices to obtain a $C^\alpha$-bound on the real and imaginary parts of the following integral: \begin{equation*}
	\begin{split}
	\int_{\Omega_\beta} \pr_{z} K_\beta(z,w) f(w) dw,
	\end{split}
	\end{equation*} defined as a principal value, both near $w = +\infty$ and $w = z$. One may compute that \begin{equation}\label{eq:P_beta}
	\begin{split}
	P_\beta(z,w) := \pr_z K_\beta(z,w) &= - \frac{\overline{w^{1/\beta}} - w^{1/\beta}}{4\pi\beta^2} \frac{z^{1/\beta - 2}}{(z^{1/\beta} - \overline{w^{1/\beta}})^2(z^{1/\beta} - w^{1/\beta})^2} \\
	&\qquad \times \left[ (1-\beta)(z^{1/\beta} - \overline{w^{1/\beta}})(z^{1/\beta} - w^{1/\beta}) - z^{1/\beta}(2z^{1/\beta} - w^{1/\beta}- \overline{w^{1/\beta}} ) \right].
	\end{split}
	\end{equation} From the remark following Lemma \ref{lem:Uniqueness_Poisson}, we know that the principal value integral \begin{equation*}
	\begin{split}
	\int_{\Omega_\beta} P_\beta(z,w) dw 
	\end{split}
	\end{equation*} is well-defined indeed equals a (complex) constant. Therefore, using the fact that $f \in C^\alpha(\overline{\Omega}_\beta)$, it suffices to bound \begin{equation*}
	\begin{split}
	\int_{\Omega_\beta} P_\beta(z,w) (f(w) - f(z))  dw 
	\end{split}
	\end{equation*}  as well as \begin{equation}\label{eq:Hoelder_diff}
	\begin{split}
	\frac{1}{|z-z'|^\alpha} \left[  \int_{\Omega_\beta} P_\beta(z,w)( f(w) - f(z)) dw - \int_{\Omega_\beta} P_\beta(z',w) (f(w) - f(z')) dw    \right]
	\end{split}
	\end{equation} for $z \ne z' \in \Omega_\beta$. We begin with the $L^\infty$ bound: we first note that, for $w, z \in \Omega_\beta$,  \begin{equation}\label{eq:P_upper}
	\begin{split}
	|P_\beta(z,w)| \le \frac{C}{|z-w|^2}.
	\end{split}
	\end{equation} for some $ C = C(\beta)$. We take $a = \min\{1,|z|\}$ and consider the regions $\{ |z - w| \le a \}$, $\{ a < |z-w| \le 2|z| \}$, and $\{ 2|z| < |z-w| \}$. In the first region, we bound the integral in absolute value as \begin{equation*}
	\begin{split}
	C\V f \V_{C^\alpha} \int_{|z-w| \le a} |z-w|^{\alpha-2} dw \le C \V f \V_{C^\alpha}. 
	\end{split}
	\end{equation*} Then, in the second region, we note that $|P_\beta(z,w)| \le |z-w|^{-2}$ and therefore we obtain a bound \begin{equation*}
	\begin{split}
	C \V f \V_{L^\infty} \int_{ a < |z-w| \le 2|z| } \frac{1}{|z-w|^2} dw \le C \ln(2 + |z|) \V f \V_{L^\infty}. 
	\end{split}
	\end{equation*} Lastly, when $\{ 2|z| < |z-w| \}$, we make a change of variables $w = |z|\zeta$ to rewrite the integral as \begin{equation*}
	\begin{split}
	\int_{\Omega_\beta \cap \{ 2 < |\frac{z}{|z|} - \zeta | \} } P_\beta(z,|z|\zeta) |z|^2d\zeta = \int_{\Omega_\beta \cap \{ 2 < |\frac{z}{|z|} - \zeta | \} } P_\beta(\frac{z}{|z|},\zeta)d\zeta  ,
	\end{split}
	\end{equation*} which is easily seen to be bounded uniformly in $z \in \Omega_\beta$, since $P_\beta(\frac{z}{|z|},\zeta)$ decays as $\zeta^{1/\beta}$ and $1/\beta > 2$. This finishes the proof of \eqref{eq:L^infty_C^alpha}. Turning to the task of obtaining the ${C}_*^\alpha$ bound, we take two points $z \ne z'$ and rewrite \eqref{eq:Hoelder_diff} as \begin{equation*}
	\begin{split}
	\frac{1}{|z-z'|^\alpha} &\left[ \int_{\Omega_\beta\cap \{ |z-w| \le 2|z-z'| \} } P_\beta(z,w)(f(w) - f(z)) dw \right. \\
	&\qquad+ \int_{\Omega_\beta\cap \{ |z-w| \le 2|z-z'| \} } P_\beta(z',w)(f(w) - f(z')) dw  \\
	&\qquad+ \int_{\Omega_\beta\cap \{ |z-w| \ge 2|z-z'| \} } \left(P_\beta(z',w) - P_\beta(z,w)\right)(f(w) - f(z)) dw \\
	&\qquad+ \left.(f(z)-f(z'))\int_{\Omega_\beta\cap \{ |z-w| \ge 2|z-z'| \} } P_{\beta}(z',w)dw \right] =: I + II + III + IV.  
	\end{split}
	\end{equation*} The integral in $I$ can be easily bounded using the upper bound \eqref{eq:P_upper}: \begin{equation*}
	\begin{split}
	\int_{\Omega_\beta\cap \{ |z-w| \le 2|z-z'| \}} |P_\beta(z,w)||f(w)-f(z)|dw \le C \V f \V_{C^\alpha} |z - z'|^\alpha,
	\end{split}
	\end{equation*} and similarly we can bound $II$ as well. We then treat the term $III$ using the mean value theorem for $\pr_z P_\beta$: note that the derivative is bounded in absolute value by \begin{equation*}
	\begin{split}
	\left| \pr_z P_\beta(z,w) \right| \le C_\beta \left(  \frac{|z|^{1/\beta -3}}{|z-w|^2 (|z|^{1/\beta-2} + |w|^{1/\beta-2})} + \frac{1}{|z-w|^3} \right),
	\end{split}
	\end{equation*} (the first term arises when $\pr_z$ falls on the factor $z^{1/\beta-2}$ in \eqref{eq:P_beta}) and hence we bound \begin{equation*}
	\begin{split}
	&\left| \int_{\Omega_\beta\cap \{ |z-w| \ge 2|z-z'| \}} \left( P_\beta(z',w) - P_\beta(z,w) \right)(f(w) - f(z)) dw  \right| \\
	&\qquad \le C \V f \V_{C^\alpha}|z - z'| \int_{|z-w| > 2|z-z'| } \left(  \frac{|z|^{1/\beta -3}}{|z-w|^2 (|z|^{1/\beta-2} + |w|^{1/\beta-2})} + \frac{1}{|z-w|^{3}} \right) |z-w|^\alpha dw \\
	&\qquad \le C \V f \V_{C^\alpha} |z - z'|^\alpha,
	\end{split}
	\end{equation*} where we have used that $\alpha < 1/\beta -2$, which ensures integrability of the kernel as $|w| \rightarrow +\infty$. Finally, it suffices to obtain a uniform bound \begin{equation*}
	\begin{split}
	\left| \int_{ \Omega_\beta \cap \{ |z-w| \ge 2|z-z'| \} } P_\beta(z,w) dw \right| \le C
	\end{split}
	\end{equation*} for all $0 < |z'| \le |z|$. Proving boundedness for this quantity is similar to a calculation done in the work of Bertozzi and Constantin on vortex patches \cite{BeCo}. It is equivalent to instead having a uniform bound for the integral of $P_\beta(z,w)$ on the region $\Omega_\beta \cap \{ |z-w| \le 2|z-z'| \}$, and after a rescaling of variables $w = |z|\zeta$, the integral equals \begin{equation*}
	\begin{split}
	\int_{\Omega_\beta \cap \{ |\frac{z}{|z|}-\zeta| \le 2|\frac{z}{|z|}- \frac{z'}{|z|}| \}}  P_\beta(\frac{z}{|z|},\zeta) d\zeta.  
	\end{split}
	\end{equation*} In other words, for all $\eta \in \Omega_\beta$ with $|\eta | = 1$, we need to bound \begin{equation*}
	\begin{split}
	\int_{ \Omega_\beta \cap B_\eta(r) } P_\beta(\eta,\zeta)d\zeta
	\end{split}
	\end{equation*} uniformly in $0 < r \le 2$, where $B_\eta(r)$ is simply the ball of radius $r$ centered at $\eta$. We make a number of reductions. First, it suffices to treat the case $r \le 1/10$ (say), since the integral of $P_\beta$ in the region $\Omega_\beta \cap (B_\eta(r)\backslash B_\eta(1/10) )$ for $r > 1/10$ can be bounded in absolute value using the simple bound \eqref{eq:P_upper}. Next, we may assume that $\eta = e^{i\mu}$ with $\mu \le \beta\pi/2$. But observe that we have a uniform bound \begin{equation*}
	\begin{split}
	|P_\beta(\eta,\zeta) - a_\beta P_1(\eta,\zeta)| \le C_\beta
	\end{split}
	\end{equation*} as long as $\eta = e^{i\mu}$ with $0 < \mu  \le \beta\pi/2$ and $|\zeta - \eta| \le 1/10$, for some constants $a_\beta$ and $C_\beta > 0$, where \begin{equation*}
	\begin{split}
	P_1(z,w) := \pr_{zz}G_1(z,w) := \pr_{zz}\left( \frac{1}{2\pi} \ln \frac{|z-w|}{|\bar{z} - w|} \right) = \frac{\overline{w} - w}{4\pi} \frac{2z - w - \overline{w}}{(z-\overline{w})^2(z-w)^2}. 
	\end{split}
	\end{equation*} Here, note that $G_1(z,w)$ is simply the Dirichlet Green's function for the upper half-plane $\mathbb{H} := \{ (r,\theta) : 0 < \theta < \pi \}$. Therefore, it suffices to bound \begin{equation*}
	\begin{split}
	\int_{ \Omega_\beta \cap B_\eta(r) } P_1(\eta,\zeta)d\zeta = \int_{\mathbb{H} \cap B_\eta(r)  } P_1(\eta,\zeta)d\zeta = \int_{\mathbb{R}^2}  P(\eta,\zeta) \tilde{\mathbf{1}}_{\mathbb{H} \cap  B_\eta(r) }(\zeta) d\zeta,
	\end{split}
	\end{equation*} where $P(z,w) := \pr_{zz} G(z,w)$ with $G(z,w) := \ln|z-w|/2\pi$, and $\tilde{\mathbf{1}}_{\mathbb{H} \cap  B_\eta(r) }$ is defined to be the odd (with respect to the $x_2 = 0$ axis) extension of the indicator function on $\mathbb{H} \cap  B_\eta(r)$.  For this final computation, we return to the real notation: writing $\eta = v_1 + iv_2$ and $\zeta = u_1 + iu_2$, we need to obtain a bound on \begin{equation*}
	\begin{split}
	\int_{\mathbb{R}^2} \frac{(v_1-u_1)(v_2-u_2)}{|v-u|^4} \tilde{\mathbf{1}}_{\mathbb{H} \cap  B_\eta(r) } du,
	\end{split}
	\end{equation*} and \begin{equation*}
	\begin{split}
	\int_{\mathbb{R}^2} \frac{(v_1-u_1)^2 - (v_2-u_2)^2}{|v-u|^4} \tilde{\mathbf{1}}_{\mathbb{H} \cap  B_\eta(r) } du,
	\end{split}
	\end{equation*} where $l \le 1/10$. We only consider the latter, since the other one can be treated in a parallel manner. We first write it as \begin{equation*}
	\begin{split}
	\int_{\mathbb{H} \cap  B_\eta(r)} \frac{(v_1-u_1)^2 - (v_2-u_2)^2}{|v-u|^4}  du- \int_{\left[\mathbb{H} \cap  B_\eta(r)\right]^-} \frac{(v_1-u_1)^2 - (v_2-u_2)^2}{|v-u|^4}  du
	\end{split}
	\end{equation*} where $\left[\mathbb{H} \cap  B_\eta(r)\right]^-$ denote the set obtained by reflecting $\mathbb{H} \cap  B_\eta(r)$ across the $x_2 = 0$ axis. To treat the first term, we use polar coordinates centered at $(v_1, v_2)$ to rewrite it as  \begin{equation*}
	\begin{split}
	\int_{v_2}^l \int_{-\pi}^{\pi} \frac{\cos(2\theta)}{r}  \mathbf{1}_{[-\theta_r,\pi + \theta_r]}   d\theta dr,
	\end{split}
	\end{equation*} where $0 < \theta_r < \pi/2 $ is  defined by $\sin(\theta_r) = v_2/r$ (There was nothing to show when $l \le v_2$). Evaluating the integral in $\theta$, \begin{equation*}
	\begin{split}
	c \int_{v_2}^l \frac{\sin(\theta_r)\cos(\theta_r)}{r} dr  \le C\int_{v_2}^{l} \frac{v_2}{r^2} dr \le C. 
	\end{split}
	\end{equation*} The other term can be shown to be uniformly bounded in a similar way: perform another explicit computation using polar coordinates in the region $\left[\mathbb{H} \cap  B_\eta(r)\right]^- \cap B_\eta(r)$, and then use the bound \eqref{eq:P_upper} in the remaining region.  This finishes the proof of the ${C}_*^\alpha$-bound \eqref{eq:dotCalpha}.
\end{proof}
	
Combining Lemmas \ref{lem:C^alpha} and \ref{lem:C^circlealpha}, we conclude that \begin{corollary}\label{cor:C^alphaintC^circlealpha}
	Let $f \in \mathring{C}^\alpha \cap C^\alpha(\overline{\Omega}_\beta)$, and $\Psi$ be the unique solution of $\Delta \Psi = f$ on $\overline{\Omega}_\beta$ with Dirichlet boundary conditions. Assume further that $0 < \alpha < \min\{ 1/\beta -2 ,1 \}$. Then, we have \begin{equation}\label{eq:C^alphaintC^circlealpha}
	\begin{split}
	\V \nabla^2\Psi\V_{ \mathring{C}^\alpha \cap C^\alpha(\overline{\Omega}_\beta) } \le C(\alpha,\beta) \V f \V_{\mathring{C}^\alpha \cap C^\alpha(\overline{\Omega}_\beta)}.
	\end{split}
	\end{equation}
\end{corollary}

\begin{remark}
	In the critical case $1/\beta - 2 = \alpha$, one has the following explicit counter-example: take $\Psi(r,\theta) = r^{1/\beta}\ln(r)\sin(\theta/\beta)\chi(r)$ where $\chi(\cdot) \in C_c^\infty([0,\infty))$ is a smooth cut-off function with $\chi(r) \equiv 1$  for $r \le 1$. Note that $\Psi$ vanishes on the boundary of $\partial\Omega_\beta$ and also note that near the corner $r = 0$, $\Delta\Psi(r,\theta) = c_\beta r^\alpha \sin(\theta/\beta) $ for some non-zero constant $c_\beta$. Clearly $\Delta\Psi \in C^\alpha(\overline{\Omega}_\beta)$ but $\partial_{x_1x_2}\Psi \notin C^\alpha(\overline{\Omega}_\beta)$. 
\end{remark}

\begin{remark}
	In the case of compactly supported function $f \in C^\alpha_c(\overline{\Omega}_\beta)$, a direct application of the Grisvard theorem \cite[Theorem 6.4.2.6]{Gris} gives that $\nabla^2\Psi \in C^\alpha(\overline{\Omega}_\beta)$, in the range $0 < \alpha < \min\{ 1/\beta -2 ,1 \}$, and there is a simple way to pass from this $C^\alpha$-estimate to a $\mathring{C}^\alpha$-estimate. To proceed, one still needs to argue that there is a unique solution to $\Delta \Psi = f$ with possibly non-decaying $f$ and that the H\"older bound is uniform in the size of the support of $f$. Therefore we have chosen to essentially re-prove this $C^\alpha$-estimate of Grisvard. 
\end{remark}

\begin{remark}
	For each fixed $k \ge 1,1 > \alpha> 0$, one has $C^{k,\alpha}$-estimates as well, assuming that the angle is sufficiently small: proceeding similarly as above, one obtains \begin{equation*}
	\begin{split}
	\V \nabla^2\Psi\V_{C^{k,\alpha}(\overline{\Omega}_\beta)} \le C(k,\alpha,\beta)\V f \V_{C^{k,\alpha}_c(\overline{\Omega}_\beta)},
	\end{split}
	\end{equation*} when $k + \alpha < 1/\beta - 2$. On the other hand, in the scale of $\mathring{C}^{k,\alpha}$-spaces, we always have \begin{equation*}
	\begin{split}
	\V \nabla^2\Psi\V_{\mathring{C}^{k,\alpha}(\overline{\Omega}_\beta)} \le C(k,\alpha,\beta)\V f \V_{\mathring{C}^{k,\alpha}(\overline{\Omega}_\beta)}
	\end{split}
	\end{equation*} without the additional assumption that $k + \alpha < 1/\beta - 2$. In particular, when $f \in \mathring{C}^\infty(\overline{\Omega}_\beta)$, we have $\nabla^2 \Psi \in \mathring{C}^\infty(\overline{\Omega}_\beta)$. We omit the details.
\end{remark}

\subsection{Local well-posedness for $2D$ Boussinesq}
In this subsection, we give a proof of local-in-time existence and uniqueness for the $2D$ Boussinesq system \eqref{B1} -- \eqref{B2} in scale-invariant H\"{o}lder spaces. We do require that $\nabla^\perp\rho_0$ and $\omega_0$ are bounded in space, but it is not required that they decay at infinity, and there are no boundary conditions for $\rho$ and $\omega$. In particular, $\rho$ and $u$ may grow linearly at infinity. 
\begin{theorem}[Local well-posedness for $2D$ Boussinesq]\label{thm:LWP2D}
Let $\omega_0$ and $\rho_0$ be functions on $\Omega$ satisfying $\omega_0, \nabla^\perp\rho_0 \in \mathring{C}^\alpha(\overline{\Omega})$. Assume that $\omega_0$ and $\partial_{x_2}\rho_0$ are odd and $\partial_{x_1}\rho_0$ is even with respect to the $x_1$-axis. Then, there exists $ T = T(\V \omega_0\V_{\mathring{C}^\alpha}, \V \nabla^\perp\rho_0 \V_{\mathring{C}^\alpha} ) > 0$ such that there is a unique solution $(\omega, \rho)$ in the class $\omega,\nabla^\perp\rho \in C([0,T);\mathring{C}^\alpha(\overline{\Omega}))$. Moreover, the solution can be continued past $T$ if and only if \begin{equation*}
\begin{split}
\int_0^T \V \nabla u_t \V_{L^\infty(\overline{\Omega})} dt < + \infty
\end{split}
\end{equation*} holds. 
\end{theorem}
\begin{remark}
As we have mentioned in the introduction, when the corner of the domain $\Omega$ is strictly less than $\pi/2$, the extra symmetry assumption on the initial data can be dropped. 
\end{remark}
\begin{proof}
For convenience, we work with the pair $(\omega,\nabla^\perp\rho)$: \begin{empheq}[left=\empheqlbrace]{align} 
&\partial_t \omega+ (u\cdot\nabla)\omega = -\partial_{x_2}\rho, \label{eq:LWP2D1} \\
\label{eq:LWP2D2} &\partial_t \nabla^\perp\rho +(u\cdot\nabla)\nabla^\perp\rho = \nabla u \nabla^\perp\rho.
\end{empheq}

\medskip

\textit{1. A priori estimates}

\medskip

Let $(\omega,\nabla^\perp\rho)$ be a smooth solution of the system \eqref{eq:LWP2D1} -- \eqref{eq:LWP2D2} defined on the time interval $[0,T]$. Then, from the bound \begin{equation*}
\begin{split}
\V \nabla u \V_{L^\infty_t\mathring{C}^\alpha} \le C \V \omega \V_{L^\infty_t\mathring{C}^\alpha},
\end{split}
\end{equation*} it follows in particular that the flow map $X(t,\cdot)$ exists as a bi-Lipschitz map of the domain. Writing equations along the flow, we obtain \begin{equation*}
\begin{split}
\begin{cases}
&\pr_t \omega\circ X = -\partial_{x_2}\rho\circ X,\\
&\pr_t \nabla^\perp \rho \circ X = \nabla u \circ X \cdot \nabla^\perp\rho\circ X. 
\end{cases}
\end{split}
\end{equation*} From this, the $L^\infty$ bounds are immediate: \begin{equation}\label{eq:Linfty_bounds}
\begin{split}
&\frac{d}{dt} \V \omega(t)\V_{L^\infty} \le \V \nabla\rho(t)\V_{L^\infty}, \\
&\frac{d}{dt} \V \nabla\rho(t)\V_{L^\infty} \le \V \nabla u(t)\V_{L^\infty} \V \nabla\rho(t)\V_{L^\infty}.
\end{split}
\end{equation} For $\mathring{C}^\alpha$-bounds, we take two points $x \ne x'$ and compute \begin{equation*}
\begin{split}
\frac{d}{dt} \left[ \frac{ |X_t(x)|^\alpha \omega_t\circ X_t(x) - |X_t(x')|^\alpha \omega_t\circ X_t(x')}{|X_t(x) - X_t(x')|^\alpha} \right] = I + II + III,
\end{split}
\end{equation*} where $I, II$, and $III$ denote the terms obtained by taking $d/dt$ on $\omega\circ X$, $X$ on the numerator, and $X$ on the denominator. Let us write for simplicity that $z = X_t(x)$ and $z' = X_t(x')$, and suppress from writing out dependence in time. Then, \begin{equation*}
\begin{split}
|I| = \left|  \frac{-|z|^\alpha \partial_{x_2}\rho(z) + |z'|^\alpha \partial_{x_2}\rho(z')}{|z-z'|^\alpha}  \right| \le \V \nabla^\perp\rho\V_{\mathring{C}^\alpha}. 
\end{split}
\end{equation*} Next, \begin{equation*}
\begin{split}
II &= \alpha \frac{z \cdot u(z)}{|z|^{2}} \frac{|z|^\alpha\omega(z)}{|z-z'|^\alpha} - \alpha \frac{z' \cdot u(z')}{|z'|^{2}} \frac{|z'|^\alpha\omega(z')}{|z-z'|^\alpha} \\
&= \alpha \frac{z\cdot u(z)}{|z|^2} \left( \frac{|z|^\alpha \omega(z) - |z'|^\alpha\omega(z')}{|z-z'|^\alpha} \right) + \alpha \left( \frac{z\cdot u(z)}{|z|^2} - \frac{z'\cdot u(z')}{|z'|^2} \right) \frac{|z'|^\alpha \omega(z')}{|z-z'|^\alpha},
\end{split}
\end{equation*} and the latter term can be further rewritten as \begin{equation*}
\begin{split}
\alpha \left(  \frac{z\cdot (u(z) - u(z')) + (z-z') \cdot u(z')}{|z|^2}  + \frac{z'\cdot u(z')}{|z|^2} - \frac{z'\cdot u(z')}{|z'|^2}  \right)\frac{|z'|^\alpha \omega(z')}{|z-z'|^\alpha}.
\end{split}
\end{equation*} Then, we may bound \begin{equation*}
\begin{split}
|II| \lesssim \left| \frac{u(z)}{z}\right|_{L^\infty} \V \omega \V_{\mathring{C}^\alpha}  +  \frac{\V \nabla u\V_{L^\infty} |z-z'|}{|z|} \frac{|z'|^\alpha |\omega|_{L^\infty}}{|z-z'|^\alpha} + \frac{|z-z'|(|z|+|z'|)}{|z|^2|z'|^2} |z'|^2 \left| \frac{u(z')}{z'}\right|_{L^\infty} \frac{|z'|^\alpha \V\omega\V_{L^\infty}}{|z-z'|^\alpha}
\end{split}
\end{equation*} and we could have assumed that $|z| \ge |z'|$, $|z| \ge |z-z'|/2$ (or switch the role of $z$ and $z'$ otherwise). This gives \begin{equation*}
\begin{split}
|II| \le C\V \omega\V_{\mathring{C}^\alpha} \V \nabla u \V_{L^\infty}. 
\end{split}
\end{equation*} Lastly, we have \begin{equation*}
\begin{split}
III = -\alpha \left[ \frac{|z|^\alpha \omega(z) - |z'|^\alpha\omega(z')}{|z-z'|^\alpha} \frac{(z-z')\cdot (u(z) - u(z'))}{|z-z'|^\alpha} \right]
\end{split}
\end{equation*} and it is easy to see that \begin{equation*}
\begin{split}
|III| \le C\V \nabla u \V_{L^\infty}\V \omega\V_{\mathring{C}^\alpha}. 
\end{split}
\end{equation*} Collecting the terms and then integrating in time, we obtain that \begin{equation*}
\begin{split}
&\left|\frac{ |X_t(x)|^\alpha \omega_t\circ X_t(x) - |X_t(x')|^\alpha \omega_t\circ X_t(x')}{|X_t(x) - X_t(x')|^\alpha}  \right|\\
&\qquad\le \left| \frac{|x|^\alpha\omega_0(x) - |x'|^\alpha\omega_0(x')}{|x-x'|^\alpha}\right| + \int_0^t \V \nabla^\perp\rho_s\V_{\mathring{C}} + C\V\nabla u_s \V_{L^\infty}\V \omega_s\V_{\mathring{C}^\alpha}ds. 
\end{split}
\end{equation*} We then obtain that \begin{equation}\label{eq:apriori_omega}
\begin{split}
\V \omega_t\V_{\mathring{C}^\alpha} \le \V \omega_0 \V_{\mathring{C}^\alpha} + C\int_0^t \V \nabla^\perp\rho_s\V_{\mathring{C}}  + \V\nabla u_s\V_{L^\infty} \V \omega_s\V_{\mathring{C}^\alpha} ds.
\end{split}
\end{equation}
Proceeding similarly for $\nabla^\perp\rho$, we obtain this time that \begin{equation}\label{eq:apriori_rho}
\begin{split}
\V \nabla^\perp\rho_t\V_{\mathring{C}^\alpha} \le \V \nabla^\perp\rho_0 \V_{\mathring{C}^\alpha} + C\int_0^t \left(\V\nabla u_s\V_{L^\infty}+ \V\nabla\rho_s\V_{L^\infty} \right) \left(\V\nabla \rho_s\V_{C^\alpha} + \V\omega_s\V_{C^{\alpha}}\right) ds.
\end{split}
\end{equation} Inequalities \eqref{eq:apriori_omega} and \eqref{eq:apriori_rho}, together with the bound $\V \nabla u\V_{L^\infty} \le C \V \omega\V_{\mathring{C}^\alpha}$  finishes an a priori estimate for $\omega$ and $\nabla\rho$; there exists a time interval $T = T(\V \omega_0\V_{\mathring{C}^\alpha},\V \nabla^\perp\rho_0\V_{\mathring{C}^\alpha}) > 0$ such that $\V \omega_t\V_{\mathring{C}^\alpha}$ and $\V\nabla^\perp\rho_t\V_{\mathring{C}^\alpha}$ remains finite in the interval $[0,T)$. 
Moreover, it is straightforward to show that $\int_0^{T} \V \nabla u_t\V_{L^\infty}dt < +\infty$ guarantees that the solution can be extended past $T$: to see this, under this assumption, first observe from the $L^\infty$-bounds \eqref{eq:Linfty_bounds} that $\V \omega\V_{L^\infty}$ and $\V \nabla^\perp \rho \V_{L^\infty}$ stays uniformly bounded up to $T$. Then, by writing $Y(t) = \V \nabla^\perp\rho_t\V_{\mathring{C}^\alpha} + \V \nabla^\perp\rho_t\V_{\mathring{C}^\alpha} $, \eqref{eq:apriori_omega} and \eqref{eq:apriori_rho} can be recast into the form \begin{equation*}
\begin{split}
Y(t) \le Y_0 + C \int_0^t Y(s)ds. 
\end{split}
\end{equation*} This guarantees that $\V \omega_t\V_{\mathring{C}^\alpha}$ and $\V \nabla^\perp\rho_t\V_{\mathring{C}^\alpha}$ stay uniformly bounded up to $T$. 

\medskip

\textit{2. Existence}

\medskip

Given initial data $\omega_0, \rho_0$ in $\mathring{C}^\alpha$, we now construct a solution, using a simple iteration scheme. We shall work in a time interval $[0,T]$, where (with a slight abuse of notation) $T > 0$ now denotes some time where $\omega,\nabla^\perp\rho$ remains uniformly bounded in $\mathring{C}^\alpha$ for $t \in [0,T]$. 
Define $u^{(0)}(t,\cdot)$ to be the velocity associated with $\omega^{(0)}(t,\cdot) := \omega_0(\cdot)$ for $t \in [0,T]$. In particular, $u^{(0)}$ is Lipschitz in space uniformly in time, so $X^{(0)}$ be the corresponding flow map, defined on $[0,T]$. 
Given $u^{(n)}$ and the associated flow $X^{(n)}$ on the time interval $[0,T]$, we may inductively define $\omega^{(n+1)}$ and $\rho^{(n+1)}$ by solving \begin{equation*}
\begin{split}
&\frac{d}{dt} \omega^{(n+1)} \circ X^{(n)} = -\partial_{x_2}\rho^{(n+1)}\circ X^{(n)}, \\
&\frac{d}{dt} \nabla^\perp\rho^{(n+1)}\circ X^{(n)}  = \nabla u^{(n)}\circ X^{(n)} \cdot \nabla^\perp\rho^{(n+1)}\circ X^{(n)}.
\end{split}
\end{equation*}
On the time interval $[0,T]$, the sequence of velocities $\{ u^{(n)} \}_{n \ge 0}$ and densities $\{ \rho^{(n)} \}_{n \ge 0}$  are uniformly bounded in $L^\infty([0,T];\mathrm{Lip}_{loc})$ and their gradients are uniformly bounded in $L^\infty([0,T];\mathring{C}^\alpha)$. In particular, by passing to a subsequence, we have \begin{equation*}
\begin{split}
u^{(n)} \rightarrow u,\quad \rho^{(n)} \rightarrow \rho,
\end{split}
\end{equation*} in $L^\infty_t \mathrm{Lip}_{loc}$, for some locally Lipschitz functions $u$ and $\rho$, satisfying $\nabla u, \nabla^\perp\rho \in L^\infty([0,T];\mathring{C}^\alpha)$. This also guarantees that the sequence of flow maps converge pointwise: $X^{(n)}(t,x) \rightarrow X(t,x)$ for each $t \in [0,T]$ and $x$, where $X$ is the flow generated by $u$. At this point, it is easy to show that $\omega = \nabla\times u$ and $\nabla^\perp\rho$ is a solution pair to the system \eqref{eq:LWP2D1} -- \eqref{eq:LWP2D2}. 

\medskip

\textit{3. Uniqueness}

\medskip

Assume that there exist two solution pairs $(\omega^1, \rho^1)$ and $(\omega^2,\rho^2)$ defined on some time interval $[0,T)$ for $T > 0$, satisfying the assumptions of the theorem and both corresponding to the initial data $(\omega_0,\rho_0)$. We denote the corresponding velocity and pressure by $(u^1, p^1)$ and $(u^2,p^2)$, respectively. Then, we simply set $\tilde{\omega} := \omega^1 - \omega^2$, $\tilde{\rho} := \rho^1 - \rho^2$, $\tilde{u} := u^1 - u^2$, and $\tilde{p} := p^1 - p^2$, which all vanish identically for $t = 0$. Taking the divergence of the both sides of the velocity equation, one obtains \begin{equation*}
\begin{split}
\Delta p^i = - \nabla \cdot \left( u^i \cdot \nabla u^i \right) - \partial_1 \rho^i = 2(\pr_1 u_1^i \pr_2 u_2^i - \pr_1 u_2^i \pr_2 u_1^i) - \partial_1 \rho^i 
\end{split}
\end{equation*} for $i = 1, 2$. Together with the Neumann boundary condition \begin{equation*}
\begin{split}
\partial_n p^i = -\rho^i n_1 
\end{split}
\end{equation*} where $n = (n_1,n_2)$, the pressure is uniquely determined. Then, we claim that the following bound holds: \begin{equation}\label{eq:pressure_bound}
\begin{split}
\left\V \frac{\nabla\tilde{p}(x)}{|x|} \right\V_{L^\infty(\Omega)} &\le C \left[ \left(\V \nabla u^1 \V_{L^\infty(\Omega)} + \V \nabla u^2 \V_{L^\infty(\Omega)}\right) \left\V \frac{\tilde{u}(x)}{|x|} \right\V_{L^\infty(\Omega)} \left( 1 + \ln\left( \frac{ \V \nabla\tilde{u} \V_{L^\infty } }{ \V |x|^{-1}\tilde{u}(x)\V_{L^\infty }} \right) \right) \right. \\
   &\qquad\qquad + \left. \left\V \frac{\tilde{\rho}(x)}{|x|} \right\V_{L^\infty(\Omega)}   \left( 1 + \ln\left( \frac{ \V \nabla\tilde{\rho} \V_{L^\infty } }{ \V |x|^{-1}\tilde{\rho}(x)\V_{L^\infty }} \right) \right) \right]
\end{split}
\end{equation} for each $t \in [0,T)$. Assuming \eqref{eq:pressure_bound} for the moment, let us complete the proof of uniqueness. Taking the difference of the velocity equations for $u^1$ and $u^2$, we obtain \begin{equation*}
\begin{split}
\pr_t \tilde{u} + u^1\cdot\nabla\tilde{u} + \tilde{u}\cdot\nabla u^2 + \nabla \tilde{p} = \begin{pmatrix}
-\tilde{\rho} \\ 0
\end{pmatrix}.
\end{split}
\end{equation*} Dividing both sides by $|x|$, composing with the flow, and taking absolute values gives \begin{equation*}
\begin{split}
\frac{d}{dt} \left\V \frac{\tilde{u}(x)}{|x|} \right\V_{L^\infty } \le  \left\V \frac{u^1(x)}{|x|} \right\V_{L^\infty }\left\V \frac{\tilde{u}(x)}{|x|} \right\V_{L^\infty } + \left\V \nabla u^2 \right\V_{L^\infty }\left\V \frac{\tilde{u}(x)}{|x|} \right\V_{L^\infty } + \left\V \frac{\nabla\tilde{p}(x)}{|x|} \right\V_{L^\infty } + \left\V \frac{\tilde{\rho}(x)}{|x|} \right\V_{L^\infty } .
\end{split}
\end{equation*} Applying the inequality \eqref{eq:pressure_bound}, we obtain  \begin{equation}\label{eq:vel_unique}
\begin{split}
\frac{d}{dt} \left\V \frac{\tilde{u}(x)}{|x|} \right\V_{L^\infty } \le  C &\left[\left\V \frac{\tilde{u}(x)}{|x|} \right\V_{L^\infty } \left( 1 + \ln\left( \frac{C}{ \V |x|^{-1} \tilde{u}(x)\V_{L^\infty} } \right) \right)  \right. \\
&\qquad+ \left. \left\V \frac{\tilde{\rho}(x)}{|x|} \right\V_{L^\infty } \left( 1 + \ln\left( \frac{C}{ \V |x|^{-1} \tilde{\rho}(x)\V_{L^\infty} } \right) \right)  \right],
\end{split}
\end{equation} where now $C > 0$ depends on $\V \nabla u^1\V_{L^\infty} + \V \nabla u^2\V_{L^\infty}$ and  $\V \nabla \rho^1\V_{L^\infty} + \V \nabla \rho^2\V_{L^\infty}$.   On the other hand, taking the difference of the density equations for $\rho^1$ and $\rho^2$, we obtain \begin{equation*}
\begin{split}
\pr_t \tilde{\rho} + u^1 \cdot\nabla\tilde{\rho} + \tilde{u}\cdot\nabla \rho^2 = 0,
\end{split}
\end{equation*} and similarly as in the above, we obtain \begin{equation}\label{eq:den_unique}
\begin{split}
\frac{d}{dt} \left\V \frac{\tilde{\rho}(x)}{|x|} \right\V_{L^\infty } \le \V \nabla u^1 \V_{L^\infty} \left\V \frac{\tilde{\rho}(x)}{|x|} \right\V_{L^\infty } + \V \nabla \rho^2\V_{L^\infty } \left\V \frac{\tilde{u}(x)}{|x|} \right\V_{L^\infty }.
\end{split}
\end{equation} From the inequalities \eqref{eq:vel_unique} and \eqref{eq:den_unique}, we see that once we set \begin{equation*}
\begin{split}
A(t):= \left\V \frac{\tilde{\rho}(x)}{|x|} \right\V_{L^\infty }, \qquad B(t) := \left\V \frac{\tilde{u}(x)}{|x|} \right\V_{L^\infty },
\end{split}
\end{equation*} then, we have \begin{equation*}
\begin{split}
\frac{d}{dt} A(t) &\le C (A(t) + B(t) ), \\
\frac{d}{dt} B(t) &\le C\left( A(t) \left(1 + \ln\left( \frac{C}{A(t)} \right) \right) + B(t)\left(1 + \ln\left( \frac{C}{B(t)} \right) \right)\right)
\end{split}
\end{equation*} for some constant $ C > 0$ on the time interval $[0,T_1]$ for any $T_1 < T$. This suffices to show that $A(t) = B(t) = 0$ on $[0,T)$ (see for instance \cite[Chap. 2]{MP}). This finishes the proof of uniqueness, and it only remains to establish \eqref{eq:pressure_bound}. 

We begin by noting that \begin{equation*}
\begin{split}
\Delta \tilde{p} = 2\left( \pr_1\tilde{u}_1 \pr_2 u_2^1 - \pr_1 u_2^1 \pr_2 \tilde{u}_1 + \pr_1 u_1^2 \pr_2 \tilde{u}_2 - \pr_1 \tilde{u}_2 \pr_2 u_1^2 \right) - \pr_1 \tilde{\rho}
\end{split}
\end{equation*} holds in $\Omega$ and $\pr_n \tilde{p} = -\tilde{\rho} n_1$ on $\partial\Omega$. Then, \begin{equation*}
\begin{split}
\tilde{p} = \int_\Omega 2G_N(x,y)\left( \pr_1\tilde{u}_1 \pr_2 u_2^1 - \pr_1 u_2^1 \pr_2 \tilde{u}_1 + \pr_1 u_1^2 \pr_2 \tilde{u}_2 - \pr_1 \tilde{u}_2 \pr_2 u_1^2 \right)(y) dy  - \int_{ \partial\Omega} G_N(x,z)\tilde{\rho}(z)n_1(z) dz,
\end{split}
\end{equation*} where $G_N(\cdot,\cdot)$ is the Neumann Green's function for the Laplacian on $\Omega$, explicitly given by \begin{equation*}
\begin{split}
G_N(x,y) = \frac{1}{2\pi}\left( \ln|x - y| + \ln |x - \hat{y}| + \ln|x+y| + \ln|x+\hat{y}| \right), \qquad \hat{y} = (y_1,-y_2).
\end{split}
\end{equation*} Since $\tilde{p}$, $\tilde{\rho}n_1$, and $\pr_1\tilde{u}_1 \pr_2 u_2^1 - \pr_1 u_2^1 \pr_2 \tilde{u}_1 + \pr_1 u_1^2 \pr_2 \tilde{u}_2 - \pr_1 \tilde{u}_2 \pr_2 u_1^2 $ are even with respect to the axis $x_2 = 0$, we may replace the kernel by \begin{equation*}
\begin{split}
\tilde{G}(x,y) := \frac{1}{2\pi}\left( \ln|x - y| + \ln |x -y^\perp| + \ln|x+y| + \ln|x+y^\perp| \right). 
\end{split}
\end{equation*} In the following, we shall use that for $|x-y| \ge 2|x|$, we have the following decay rates in $y$: \begin{equation*}
\begin{split}
\left| \nabla_x\tilde{G}(x,y) \right| \le  \frac{C|x|^2}{|y|^3}  ,\qquad \left| \nabla_x^2\tilde{G}(x,y) \right| \le  \frac{C|x|^2}{|y|^4}
\end{split}
\end{equation*} for some absolute constant $C > 0$. Differentiating the expression for $\tilde{p}$, we obtain \begin{equation*}
\begin{split}
\nabla\tilde{p} = \int_{ \Omega} 2 \nabla_x\tilde{G}(x,y) \left( \pr_1\tilde{u}_1 \pr_2 u_2^1 - \pr_1 u_2^1 \pr_2 \tilde{u}_1 + \pr_1 u_1^2 \pr_2 \tilde{u}_2 - \pr_1 \tilde{u}_2 \pr_2 u_1^2\right)(y) dy - \int_{\partial\Omega} \nabla_x \tilde{G}(x,z)\tilde{\rho}(z)n_1(z) dz .
\end{split}
\end{equation*} We re-write the second term as \begin{equation*}
\begin{split}
\left[\int_{\partial\Omega \cap \{ |x - z | \le \epsilon|x| \}} + \int_{\partial\Omega \cap \{ \epsilon|x|< |x - z | \le 2|x| \}} + \int_{\partial\Omega \cap \{ |x - z | > 2|x| \}} \right]  \nabla_x \tilde{G}(x,z)\tilde{\rho}(z)n_1(z) dz 
\end{split}
\end{equation*} for some $\epsilon \le 1$ to be chosen below, and then the second piece is bounded in absolute value simply by \begin{equation*}
\begin{split}
\int_{\partial\Omega \cap \{ \epsilon |x| <  |x - z | \le 2|x| \}} \frac{|z|}{|x-z|} \cdot \frac{|\tilde{\rho}(z)|}{|z|} dz \le C|x| \ln\left(\frac{2}{\epsilon}\right) \left\V \frac{\tilde{\rho}(z)}{|z|} \right\V_{L^\infty(\Omega)},
\end{split}
\end{equation*} and the last part by \begin{equation*}
\begin{split}
C\int_{\partial\Omega \cap \{ |x - z | > 2|x| \}} \frac{|z||x|}{|z|^3} \cdot \frac{|\tilde{\rho}(z)|}{|z|} dz \le C|x| \left\V \frac{\tilde{\rho}(z)}{|z|} \right\V_{L^\infty(\Omega)}.
\end{split}
\end{equation*} In the first term, we may subtract \begin{equation*}
\begin{split}
\tilde{\rho}(x) \int_{\partial\Omega \cap \{ |x - z | \le \epsilon|x| \}} \nabla_x \tilde{G}(x,z) n_1(z)dz 
\end{split}
\end{equation*} (which is bounded in absolute value by a constant multiple of $|x| \V |z|^{-1}\tilde{\rho}(z)\V_{L^\infty }$) to obtain a bound \begin{equation*}
\begin{split}
\int_{\partial\Omega \cap \{ |x - z | \le \epsilon|x| \}}C |x-z|  |\nabla_x \tilde{G}(x,z)   | |\nabla\tilde{\rho}(z)| dz \le C|x|\epsilon \V \nabla \tilde{\rho}\V_{L^\infty } .
\end{split}
\end{equation*} Collecting the bounds, and choosing $\epsilon = \frac{  \V |z|^{-1}\tilde{\rho}(z)\V_{L^\infty } }{2\V \nabla\tilde{\rho} \V_{L^\infty }}  < 1$ gives  \begin{equation*}
\begin{split}
\frac{1}{|x|}\left| \int_{\partial\Omega} \nabla_x \tilde{G}(x,z)\tilde{\rho}(z)n_1(z) dz  \right| \le C\left\V \frac{\tilde{\rho}(z)}{|z|} \right\V_{L^\infty(\Omega)} \left( 1 + \ln\left( \frac{ \V \nabla\tilde{\rho} \V_{L^\infty } }{ \V |z|^{-1}\tilde{\rho}(z)\V_{L^\infty }} \right) \right).
\end{split}
\end{equation*} Now we move on to the first term in the expression for $\nabla\tilde{p}$. Similarly as in the above, we begin by splitting the integral to regions $\int_{\Omega \cap \{ |x-y| > \epsilon|x| \} }$ and $\int_{ \Omega \cap \{ |x-y| \le \epsilon|x| \} } $, where $0 < \epsilon \le 1$ will be a constant to be chosen below. In the latter region, we simply use the gradient bounds for velocities to obtain a bound \begin{equation*}
\begin{split}
&\left| \int_{ \Omega\cap \{ |x-y| \le \epsilon|x| \}} 2 \nabla_x\tilde{G}(x,y) \left( \pr_1\tilde{u}_1 \pr_2 u_2^1 - \pr_1 u_2^1 \pr_2 \tilde{u}_1 + \pr_1 u_1^2 \pr_2 \tilde{u}_2 - \pr_1 \tilde{u}_2 \pr_2 u_1^2\right)(y) dy\right| \\
&\qquad \le C\epsilon|x|   \left(\V\nabla u^1 \V_{L^\infty} + \V \nabla u^2\V_{L^\infty}\right)  \V \nabla \tilde{u}\V_{L^\infty}.
\end{split}
\end{equation*} Then, in the other region, we integrate by parts to obtain \begin{equation*}
\begin{split}
-2\int_{\Omega \cap \{ |x-y| > \epsilon|x| \} } &\left(\partial_{x_1} \nabla_x \tilde{G}(x,y) \pr_2 u_2^1(y) - \partial_{x_2}\nabla_x \tilde{G}(x,y) \partial_1 u_2^1(y) \right) \tilde{u}_1(y) \\
 &\qquad + \left( \pr_{x_2}\nabla_x\tilde{G}(x,y)\pr_1 u_1^2(y) - \partial_{x_1} \nabla_x \tilde{G}(x,y) \pr_2 u_1^2(y)\right) \tilde{u}_2(y) dy 
\end{split}
\end{equation*} together with the boundary term which is bounded in absolute value by $$C( \V \nabla u^1 \V_{L^\infty } + \V \nabla u^2 \V_{L^\infty}) \left\V \frac{\tilde{u}(y)}{|y|}\right\V_{L^\infty }.$$ Then, the above integral can be bounded by considering separately the regions $\{ \epsilon|x| \le  |x-y | \le 2|x| \}$ and $\{ 2|x| < |x-y| \}$. In the former region, we simply bound \begin{equation*}
\begin{split}
\left| \int_{\Omega \cap \{ \epsilon|x| \le  |x-y | \le 2|x| }   \partial_{x_1} \nabla_x \tilde{G}(x,y) \pr_2 u_2^1(y) \tilde{u}_1(y) dy \right| &\le C\int_{\Omega \cap \{ \epsilon|x| \le  |x-y | \le 2|x| } \frac{|y|}{|x-y|^2} |\partial_2 u_2^1(y)| \frac{|\tilde{u}_1(y)|}{|y|} dy  \\
&\le C|x| \ln\left( \frac{2}{\epsilon} \right) \V \nabla u^1 \V_{L^\infty}  \left\V \frac{\tilde{u}(y)}{|y|} \right\V_{L^\infty} . 
\end{split}
\end{equation*} In the case $2|x| < |x-y|$, we use decay of $\nabla_x^2 \tilde{G}$ in $y$ to obtain a bound  of the form $C|x| \V \nabla u^1\V_{L^\infty } \V |y|^{-1} \tilde{u}(y)\V_{L^\infty}$. Collecting the bounds and choosing $\epsilon = \frac{ \V |y|^{-1} \tilde{u}(y)\V_{L^\infty} }{2 \V \nabla \tilde{u}\V_{L^\infty}} < 1$ gives \begin{equation*}
\begin{split}
&\frac{1}{|x|} \left| \int_{ \Omega} 2 \nabla_x\tilde{G}(x,y) \left( \pr_1\tilde{u}_1 \pr_2 u_2^1 - \pr_1 u_2^1 \pr_2 \tilde{u}_1 + \pr_1 u_1^2 \pr_2 \tilde{u}_2 - \pr_1 \tilde{u}_2 \pr_2 u_1^2\right)(y) dy \right| \\
&\qquad\le C \left(  \V \nabla u^1 \V_{L^\infty}  + \V \nabla u^2 \V_{L^\infty}   \right) \left\V \frac{\tilde{u}(y)}{|y|} \right\V_{L^\infty } \left( 1 + \ln\left( \frac{ \V \nabla\tilde{u} \V_{L^\infty } }{ \V |y|^{-1}\tilde{u}(y)\V_{L^\infty }} \right) \right).
\end{split}
\end{equation*} This finishes the proof of \eqref{eq:pressure_bound}.


\end{proof}
\subsection{Local well-posedness for the $1D$ system describing homogeneous data}
\begin{proposition}[Local well-posedness for the $1D$ system]\label{prop:LWP1D}
Let $(g_0,P_0)$ satisfy the following assumptions: for some $k \ge 0$ and $0 \le \alpha \le 1$, let $g_0 \in C^{k,\alpha}[-\pi/4,\pi/4]$ be odd and $P_0 \in C^{k+1,\alpha}[-\pi/4,\pi/4]$ even. Then, there exists some $T = T( \V g_0 \V_{C^{k,\alpha}},\V P_0\V_{C^{k+1,\alpha}}  ) > 0$, so that there exists a unique solution $(g,P)$ to the system \eqref{BSI1} -- \eqref{BSI2} in the class $(g,P) \in C([0,T); C^{k,\alpha} \times C^{k+1,\alpha})  $. Moreover, the solution can be continued past $T$ if and only if \begin{equation*}
\begin{split}
\int_0^T \V g_t\V_{L^\infty} dt < + \infty
\end{split}
\end{equation*} holds. 
\end{proposition}
\begin{proof}
We only sketch the proof, as it is strictly analogous to the arguments for the local well-posedness for $2D$ Boussinesq.
We take $k = 0$ and $0 < \alpha <1$ for simplicity. Then, from the equation for $g$, we obtain that \begin{equation*}
\begin{split}
\V g_t\V_{C^\alpha} \le \V g_0\V_{C^\alpha} + C\int_0^t \V P_s\V_{C^{1,\alpha}} + \V G'_s\V_{L^\infty} \V g_s\V_{C^\alpha} ds. 
\end{split}
\end{equation*} Next, from the $P$ equation, it follows that \begin{equation*}
\begin{split}
\V P_t \V_{L^\infty} \le \V P_0\V_{L^\infty} + \int_0^t \V G'_s\V_{L^\infty} \V P_s\V_{L^\infty}ds . 
\end{split}
\end{equation*} Then, from \begin{equation*}
\begin{split}
\pr_t P' + 2G \partial_{\theta}P' = G''P - G'P', 
\end{split}
\end{equation*} we obtain \begin{equation*}
\begin{split}
\V P_t'\V_{L^\infty} \le \V P_0\V_{\mathrm{Lip}} + \int_0^t \V G'_s\V_{\mathrm{Lip}}\V P_s\V_{\mathrm{Lip}}ds
\end{split}
\end{equation*} as well as \begin{equation*}
\begin{split}
\V P'_t\V_{C^\alpha} \le \V P_0\V_{C^{1,\alpha}} + C \int_0^t \left( \V G''_s\V_{L^\infty} + \V P_s\V_{\mathrm{Lip}} \right) \V P'_s\V_{C^\alpha}ds.
\end{split}
\end{equation*} From $\V G''\V_{L^\infty} \le C \V g \V_{L^\infty}$, one obtains a priori estimates as well as the proof of the fact that $\V g_t\V_{L^\infty}$ controls blow-up. 
\end{proof}
We supply a proof of a simple $L^\infty$-stability result for $C^1 \times C^2$ solutions for the $1D$ system. 
\begin{lemma}\label{lem:stability}
Let $(\overline{g},\overline{P})$ be a solution to \eqref{BSI1} -- \eqref{BSI2} belonging to $C^0([0,T];C^{k+1}\times C^{k+2})  $, with initial data $(\overline{g}_0, \overline{P}_0)$. Let $(g,P)$ be another $C^{k+1} \times C^{k+2}$ solution with initial data $(g_0,P_0)$ on the same time interval, so that $\V g_0 -\overline{g}_0\V_{C^{k+1}} + \V P_0 - \overline{P}_0 \V_{C^{k+2}} < 1/2$. Then, there exists some $0 <T_1 \le T$ depending only on $k$ and the $C^{k+1} \times C^{k+2}$-norm of $(\overline{g}_0, \overline{P}_0)$ such that \begin{equation*}
\begin{split}
\V g -\overline{g}\V_{C^k} + \V P - \overline{P} \V_{C^{k+1}} <2\left(\V g_0 -\overline{g}_0\V_{C^{k+1}} + \V P_0 - \overline{P}_0 \V_{C^{k+2}}\right).
\end{split}
\end{equation*}
\end{lemma}
\begin{proof}
We only consider the case $k = 0$, and write $\mu := g - \overline{g}$ as well as $p:= P - \overline{P}$. Then, from \begin{equation*}
\begin{split}
\pr_t \mu + 2G \pr_{\theta}\mu = p\sin\theta + p'\cos\theta - 2(G-\overline{G})\overline{g}'
\end{split}
\end{equation*} (where $G'' + 4G = g$ and $\overline{G}'' + 4\overline{G} = \overline{g}$) one obtains \begin{equation}\label{eq:o_stability}
\begin{split}
\V \mu_t\V_{L^\infty}  \le \V \mu_0\V_{L^\infty} + C(\V\overline{g}_0\V_{C^1}) \int_0^t \left( \V p_s\V_{\mathrm{Lip}} + \V \mu_s\V_{L^\infty} \right) ds.
\end{split}
\end{equation} Similarly, we have \begin{equation*}
\begin{split}
\pr_t p + 2G \partial_{\theta} p = - 2(G - \overline{G})\overline{P}' -(\overline{G} - G)'\overline{P} + \overline{G}'p + (G - \overline{G})'p.
\end{split}
\end{equation*} Working with the above equation for $p$ and that for $p'$, one obtains \begin{equation}\label{eq:p_stability}
\begin{split}
\V p_t\V_{\mathrm{Lip}} \le \V p_0\V_{\mathrm{Lip}} + C(\V\overline{g}_0\V_{C^1},\V \overline{P}_0\V_{C^2}) \int_0^t \V p_s\V_{\mathrm{Lip}} + \V \mu_s\V_{L^\infty} +  \V p_s\V_{\mathrm{Lip}}\V \mu_s\V_{L^\infty}ds.
\end{split}
\end{equation} The inequalities \eqref{eq:o_stability} and \eqref{eq:p_stability} finish the proof.
\end{proof}
\section{Analysis of the $1D$ system}\label{sec:1D}
In this section, we analyze the system solved by scale-invariant solutions to the Boussinesq equation on the interval $[-\pi/4,\pi/4]$:
\begin{empheq}[left=\empheqlbrace]{align} 
&\partial_t g+2G\partial_\theta g =\sin\theta P+\cos\theta\partial_\theta P,\label{BSI1} \\
\label{BSI2} &\partial_t P+2G\partial_\theta P =P\partial_\theta G,
\end{empheq} where $G$ is obtained from $g$ by solving $\partial_{\theta\theta}G + 4G = g$. 
We shall always assume that the initial data $(g_0,P_0)$ is odd and even respectively across zero, which is preserved by the dynamics. 
As it was discussed in the introduction, if one considers zero density initial data $P_0 \equiv 0$, then the system reduces to the $1D$ equation for scale-invariant solutions to $2D$ Euler, which is globally well-posed due to the $L^\infty$ conservation of $g$ (see \cite[Section 3]{EJSI} where a number of properties of this system are investigated). On the other hand, we show in this section that smooth initial data can blow up in finite time for the system \eqref{BSI1} -- \eqref{BSI2}. A few sign assumptions on the initial data are made for our blow-up proof to work, but it is expected that almost all initial data blows up, unless $P_0$ is trivial. 
Let us now state the blow-up result for the $1D$ system. 
\begin{theorem}\label{thm:1Dblowup}
Assume that the initial data $(g_0, P_0) \in C^\infty([-\pi/4,\pi/4]) \times C^\infty([-\pi/4,\pi/4]) $ satisfy \begin{enumerate}
\item $g_0$ is odd and satisfies $g_0, g_0' \ge 0$ on $[0,\pi/4]$. 
\item $P_0$ is even and satisfies $P_0, P_0', P_0 + P_0'' \ge 0$ on $[0,\pi/4]$. 
\item $P_0(\pi/4) > \sqrt{2} P_0(0)$.
\end{enumerate}
Then, the unique smooth solution $(g,P)$ to the system \eqref{BSI1} -- \eqref{BSI2} blows up in finite time. 
\end{theorem}
For a simple example of initial data satisfying the assumptions, one may take $g_0 \equiv 0$ and $P_0(\theta) = \theta^2$. The key observations we need to make about this system is that solutions satisfy some miraculous monotonicity properties which seem to be very difficult to deduce in the full two-dimensional system. 
Let us first derive the equations \eqref{BSI1} -- \eqref{BSI2} from the $2D$ Boussinesq system. We make the ansatz that the vorticity and the density are radially homogeneous with degree zero and one, respectively: using polar coordinates,  \begin{equation}\label{eq:ansatz}
\begin{cases}
\omega(r,\theta) = g(\theta),\\
\rho(r,\theta) = r P(\theta). 
\end{cases}
\end{equation} Then, writing the stream function $\Psi$ associated with $\omega$ as $\Psi(r,\theta) = r^2 G(\theta)$, we have \begin{equation}\label{eq:velocity}
\begin{split}
u(r,\theta) = \nabla^\perp\Psi(r,\theta) = 2rG(\theta) \hat{\theta} - r \partial_{\theta} G(\theta) \hat{r}.
\end{split}
\end{equation} Here, \begin{equation*}
\begin{split}
\hat{\theta} = 
\frac{x^\perp}{|x|}, \qquad \hat{r} = \frac{x}{|x|}.
\end{split}
\end{equation*} Moreover, $\Delta \Psi = \omega$ gives that \begin{equation}\label{eq:elliptic}
\begin{split}
\partial_{\theta\theta} G + 4G = g. 
\end{split}
\end{equation} Plugging in relations \eqref{eq:ansatz} and \eqref{eq:velocity} into the $2D$ Boussinesq system in vorticity form results in equations \eqref{BSI1} and \eqref{BSI2}. We have used the formulas \begin{equation*}
\begin{split}
\nabla = \hat{r} \pr_r  + \hat{\theta} \frac{\partial_{\theta}}{r},\qquad \nabla^\perp = -\hat{r}\frac{\pr_{\theta}}{r} + \hat{\theta} \pr_r, \qquad \Delta = \pr_{rr} + \frac{\pr_r}{r} + \frac{\partial_{\theta\theta}}{r^2}.
\end{split}
\end{equation*}
\subsection{Positivity and monotonicity lemmas}
We begin by recalling a result from \cite[Lemma 3.7]{EJSI} which says that the velocity has a sign in the ``odd and positive'' scenario. 
\begin{lemma}[see \cite{EJSI}]\label{lem:PropertiesOfG}
For some $k \ge 0$ and $0 \le \alpha \le 1$, suppose $g$ is $C^{k,\alpha}$ and odd on $[-\pi/4,\pi/4]$, non-negative in $[0,\pi/4]$, and that $G$ is the unique odd $C^{k+2,\alpha}$ solution to the boundary value problem: \begin{equation*}
\begin{split}
\begin{cases}
&G'' + 4G = g \\
&G(\pm\pi/4) = 0.
\end{cases}
\end{split}
\end{equation*} Then, 
\begin{enumerate}
\item $G$ is non-positive on $[0,\pi/4]$. 
\item $G''\geq g$ on $[0,\pi/4]$.
\item We have the following formula for $G'$: \begin{equation}\label{G'} G'(\theta)= \sin(2\theta) \int_0^\theta g(\theta')\sin(2\theta')d\theta' -\cos(2\theta)\int_{\theta}^{\pi/4} g(\theta')\cos(2\theta')d\theta'\end{equation} on $[0,\pi/4]$.
\end{enumerate}
\end{lemma}
\begin{proof}
Under the odd assumption in $g$, it is straightforward to show that the unique odd solution is given by \begin{equation}\label{eq:kernel_G}
\begin{split}
G(\theta) &= \frac{1}{4} \int_0^{\pi/4} \left( |\sin(2\theta - 2\theta')| - |\sin(2\theta+2\theta')| \right)g(\theta')d\theta'  ,\qquad \theta \in [0,\pi/4].
\end{split}
\end{equation} The kernel is non-positive for $\theta,\theta' \in [0,\pi/4]$, and it follows that $G \le 0$. Hence $G'' = g - 4G \ge 0$. The explicit representation formula \eqref{G'} for $G'$ can be obtained by directly differentiating \eqref{eq:kernel_G}. 
\end{proof}
\begin{lemma}\label{lem:positivity}
Let $g, P \in C^3([0,T]\times [-\pi/4,\pi/4])$ be a solution pair to the system \eqref{BSI1} -- \eqref{BSI2}. Assume that initially we have \begin{equation*}
\begin{split}
g_0, \quad \partial_{\theta} g_0 \ge 0 \qquad \mbox{and}\qquad P_0, \quad \partial_{\theta} P_0, \quad P_0  + \partial_{\theta\theta}^2 P_0 \ge 0.
\end{split}
\end{equation*} Then the solution keeps these signs for all $t \in [0,T]$; that is, \begin{equation*}
\begin{split}
g, \quad \partial_{\theta} g \ge 0 \qquad \mbox{and}\qquad P, \quad \partial_{\theta} P, \quad P  + \partial_{\theta\theta}^2 P \ge 0.
\end{split}
\end{equation*}
\end{lemma}
\begin{proof}
We break up the proof into three steps: 

\vspace{3mm}

\emph{1. Positivity of $P$}

\vspace{3mm}

\noindent The fact that $P\geq 0$ on $[0,T]$ follows from the following lower bound on $P$ which can be easily seen from \eqref{BSI2}:
$$\inf P(t) \geq \inf P_0 \cdot \exp\left(-\int_0^t |\partial_\theta G|_{L^\infty}(s)ds\right) \geq 0,$$ since $G$ is assumed to be $C^3$ and since $\inf P_0\geq 0.$ 

\vspace{3mm}

\emph{2. Positivity of $\partial_\theta P$ and $g$}

\vspace{3mm}

\noindent Next, we analyze the equations for $\partial_\theta P$ and $g$ simultaneously: \begin{equation*}
\begin{split}
\begin{cases}
&\partial_t g +2G\partial_\theta g =\sin(\theta) P+ \cos(\theta) \partial_\theta P \\
&\partial_t \partial_\theta P +2G \partial_\theta \partial_\theta P=P\partial_{\theta\theta} G-\partial_\theta P\partial_\theta G
\end{cases}
\end{split}
\end{equation*}
Now we rewrite the equations along the characteristics $X(t,\cdot)$ associated to $2G$ to get rid of the transport term:
\begin{equation*}
\begin{split}
\begin{cases}
& \partial_t(g\circ X)=\sin(X)P\circ X + \cos(X) \partial_\theta P\circ X \\
& \partial_t (\partial_\theta P\circ X)=P\circ X \partial_{\theta\theta} G \circ X-\partial_\theta P\circ X\partial_\theta G\circ X.
\end{cases}
\end{split}
\end{equation*}
Now recall from Lemma \ref{lem:PropertiesOfG} that $\partial_{\theta\theta}G\geq g$ so long as $g\geq 0$. Unfortunately, it is not clear how to use a pure ODE argument to propagate positivity of $g$ without an equation on $\partial_{\theta\theta}G$. Hence, one can make the following further assumption on $g$ and $\partial_\theta P$ at $t = 0$: 
$$\frac{g_0(\theta)}{\sin(\theta)},\quad \frac{\partial_\theta P_0(\theta)}{\sin(\theta)}\geq \epsilon.$$ We will eventually get rid of this assumption by sending $\epsilon\rightarrow 0$. Considering the equation for $\frac{g}{\sin(X^{-1})}$ and $\frac{\partial_\theta P}{\sin(X^{-1})}$ we see that \begin{equation}\label{eq:goversin}
\begin{split}
& \frac{d}{dt} \left( \frac{g}{\sin(X^{-1})} \right) \circ X = \frac{\sin(X)}{\sin\theta} \cdot P\circ X + \cos(X) \cdot \left( \frac{P'}{\sin(X^{-1})} \right) \circ X
\end{split}
\end{equation} and \begin{equation}\label{eq:Pprimeoversin}
\begin{split}
& \frac{d}{dt} \left( \frac{P'}{\sin(X^{-1})} \right) \circ X = P\circ X \cdot  \frac{G''(X)}{\sin\theta} - G'\circ X \cdot  \left( \frac{P'}{\sin(X^{-1})} \right) \circ X
\end{split}
\end{equation} holds. Here, we have used the fact that the velocity at $\theta = 0$ is always zero by symmetry, and therefore $X(t,0) = 0$. Since the solution is smooth in space and time, there exists some small $0 < t^* \le T$ (depending only on $C^1$ norm of $g$ and $P$) so that \begin{equation*}
\begin{split}
\frac{g(t,\theta)}{\sin(X^{-1}_t(\theta))}, \quad \frac{P'(t,\theta)}{\sin(X^{-1}_t(\theta))} \ge \frac{\epsilon}{2}, \qquad t \in [0,t^*]. 
\end{split}
\end{equation*} Then, on this time interval, we have $G'' \ge 0$, and since $P \ge 0$, from equations \eqref{eq:goversin} and \eqref{eq:Pprimeoversin}, it follows that on the same time interval $[0,t^*]$, we indeed have \begin{equation}\label{eq:positivity_lower_bound}
\begin{split}
\frac{g}{\sin(X^{-1})}, \quad \frac{P'}{\sin(X^{-1})} \ge \epsilon \exp\left( -Ct \right),
\end{split}
\end{equation} 
for some $C$ depending on $P$ and $g$ in the $C^1$ norm. Using a continuation argument, one can actually show that, \eqref{eq:positivity_lower_bound} is valid on the full time interval $[0,T]$; that is, we split the interval $[0,T]$ into $\cup_{k=0}^{K} \left([kt^*/2,(k+1)t^*/2] \cap [0,T]\right) $ and then show \eqref{eq:positivity_lower_bound} inductively in $k$. 
Now we may send $\epsilon\rightarrow 0$ to recover the general case, using continuity of the solution operator $(g_0,P_0) \mapsto (g_t,P_t)$ in the $C^1 \times C^2$-topology for smooth ($C^3$ is enough) solutions (see Lemma \ref{lem:stability}). This establishes that $P', g \ge 0$. 

\vspace{3mm}

\emph{3. Positivity of $P + P''$ and $g'.$}

\vspace{3mm}

\noindent A direct computation gives us that $g'$ and $ P +P''$ satisfy the following system: 
\begin{equation}\label{eq:P''andg'}
\begin{split}
\begin{cases}
&\pr_t g' + 2G\pr_{\theta}g' = - 2G'g' + \cos\theta (P + P'')\\
&\pr_t (P + P'') + 2G\pr_{\theta}(P+P'') = Pg' - 3G'(P + P''). 
\end{cases}
\end{split}
\end{equation} We have used that $g' = G''' + 4G'$. Now we write $A:=g'\circ X$ and $B:=(P+P'')\circ X$ and we see: \begin{equation*}
\begin{split}
\begin{cases}
&\pr_t A = -2G'(X) A + \cos(X)B,\\
&\pr_t B = P(X) A - 3G'(X) B. 
\end{cases}
\end{split}
\end{equation*} This ODE system clearly shows that $A$ and $B$ remains non-negative if initially so. This finishes the proof. \end{proof}
\begin{remark}
If one replaces the assumption that $P_0 + P_0'' \ge 0$ with $P_0'' \ge 0$ while keeping the others in Lemma \ref{lem:positivity}, it is not true in general that $P ''$ stays non-negative for $t > 0$. Hence one really needs to work with $P + P'' \ge 0$, which is sufficient to keep $g' \ge 0$. 
\end{remark}
\subsection{Proof of finite-time blow up}
Now we are in a position to complete the proof of Theorem \ref{thm:1Dblowup}. 
\begin{proof}[Proof of Theorem \ref{thm:1Dblowup}]
Let $(g_0,P_0)$ be a pair of initial data satisfying the assumptions of Theorem \ref{thm:1Dblowup}. For the sake of contradiction, assume that the solution remains smooth for arbitrarily large $T > 0$. We first note that from Lemma \ref{lem:positivity}, $g, g' \ge 0$ and $P, P', P + P'' \ge 0$ for all time $t \ge 0$ and $\theta \in [0,\pi/4]$. Also recall that we always have $g$ odd and $P$ even.
Integrating the equation for $g$ \eqref{BSI1}, we obtain \begin{equation}\label{eq:g_integral}
\begin{split}
\frac{d}{dt} \int_0^{\pi/4} g(t,\theta)d\theta = 2 \int_{0}^{\pi/4} \sin\theta P(t,\theta)d\theta + \frac{1}{\sqrt{2}}P(t,\pi/4) - P(t,0) + G'(t,\pi/4)^2 - G'(t,0)^2. 
\end{split}
\end{equation} On the other hand, evaluating \eqref{BSI2} at $\pi/4$ and $0$ gives \begin{equation*}
\begin{split}
\frac{d}{dt} P(t,\pi/4) = G'(t,\pi/4)P(t,\pi/4),\qquad \frac{d}{dt} P(t,0) = G'(t,0)P(t,0). 
\end{split}
\end{equation*} The formula \eqref{G'} from Lemma \ref{lem:PropertiesOfG} at $\pi/4$ and $0$ gives respectively, \begin{equation*}
\begin{split}
G'(t,\pi/4) = \int_0^{\pi/4} g(t,\theta) \sin(2\theta)d\theta, \qquad G'(t,0) = -\int_0^{\pi/4} g(t,\theta)\cos(2\theta)d\theta. 
\end{split}
\end{equation*} In particular, since $g \ge 0$, $G'(t,\pi/4) \ge 0 \ge G'(t,0)$, and $P(t,\pi/4) \ge P_0(\pi/4) > \sqrt{2}P_0(0) \ge \sqrt{2}P(t,0)$. Next, since $\sin(2(\pi/4-\theta)) = \cos(2\theta)$ and $g$ is monotonically increasing in $[0,\pi/4]$, we have that \begin{equation*}
\begin{split}
|G'(t,\pi/4)| \ge |G'(t,0)|, 
\end{split}
\end{equation*} as well as \begin{equation}\label{eq:G'versusg}
\begin{split}
G'(t,\pi/4) \ge c \int_0^{\pi/4} g(t,\theta)d\theta,
\end{split}
\end{equation} because at least half of the mass of $g$ is present in $[\pi/8,\pi/4]$, where $\sin(2\theta) \ge 1/\sqrt{2}$. 
Returning to \eqref{eq:g_integral}, we now see that \begin{equation}\label{eq:g_ODE}
\begin{split}
\frac{d}{dt} \int_0^{\pi/4} g(t,\theta)d\theta \ge \frac{1}{\sqrt{2}} P(t,\pi/4) - P_0(0) \ge c P(t,\pi/4) ,
\end{split}
\end{equation} for some absolute constant $c > 0$ depending only on $P_0$. Then, now using \eqref{eq:G'versusg}, we see that \begin{equation}\label{eq:P_ODE}
\begin{split}
\frac{d}{dt} P(t,\pi/4) \ge c P(t,\pi/4) \cdot \int_0^{\pi/4} g(t,\theta)d\theta
\end{split}
\end{equation} holds. Given inequalities \eqref{eq:g_ODE} and \eqref{eq:P_ODE}, we may finish the proof as follows: consider the ODE system \begin{equation}\label{eq:ODE}
\begin{split}
\frac{d}{dt} A(t) = cB(t),\qquad \frac{d}{dt} B(t) = cB(t)A(t),
\end{split}
\end{equation} where $A(0) := \int_0^{\pi/4} g_0(\theta)d\theta > 0$ and $B(0) := P_0(\pi/4) > 0$. Then, by defining $C(t) := B(t)^{1/2}$, we have \begin{equation*}
\begin{split}
\frac{d}{dt} A(t) = c C(t)^2, \qquad \frac{d}{dt} C(t) = \frac{c}{2}C(t)A(t).
\end{split}
\end{equation*} Then, \begin{equation*}
\begin{split}
\frac{d}{dt}\left( A(t)C(t) \right) = cC(t)^3 + \frac{c}{2} C(t)A(t)^2 = \left( A(t)C(t) \right)^{3/2} \left( c \cdot \frac{C(t)^{3/2}}{A(t)^{3/2}} + \frac{c}{2} \cdot \frac{A(t)^{1/2}}{C(t)^{1/2}}  \right).
\end{split}
\end{equation*} Since \begin{equation*}
\begin{split}
c \cdot \frac{C(t)^{3/2}}{A(t)^{3/2}} + \frac{c}{2} \cdot \frac{A(t)^{1/2}}{C(t)^{1/2}} \ge c'
\end{split}
\end{equation*} for some absolute constant $c' > 0$, $A(t)C(t)$ must become infinite in finite time. By a comparison argument, it shows that \begin{equation*}
\begin{split}
\int_0^{\pi/4} g(t,\theta)d\theta \cdot (P(t,\pi/4))^{1/2}
\end{split}
\end{equation*} becomes infinite in finite time as well. This is a contradiction to the assumption that the solution is global in time.
\end{proof}
\subsection{Blow-up for the $1D$ system in $[-L,L]$ for any $L<\pi/2$.}
Now we explain how the previous proof can be modified to give blow-up when the angle of the corner is larger than $\pi/4$. In fact, the exact same proof works except that we need to establish a few lemmas which we used in the case $L=\pi/4.$ Upon inspecting the proof, we see that all we need to establish which doesn't carry over word for word are the following statements:
\begin{enumerate}
\item If $g\geq 0$, then $G''\geq g$.
\item If $g,g'\geq 0$ then $|G'(L)|\geq |G'(0)|$.
\item If $g,g'\geq 0$ then $|G'(L)|\geq c|g|_{L^1}$ for some small constant $c>0$. 
\end{enumerate}
In the case $L=\pi/4$, we had formula \eqref{G'} which gave us $G$ from $g$, and we used this formula to establish the above statements. Here we show how to establish these statements without a solution formula, and the advantage is that it allows us to handle the case of general $L<\pi/2$.  
First notice that there is a unique odd $G\in H^2([-L,L])$ solving the elliptic boundary value problem \begin{equation}\label{ellipticproblem} 4G+G''=g, \qquad G(-L)=G(L)=0\end{equation}  whenever $g\in L^2([-L,L])$ is odd and $0<L< {\pi}/{2}.$ This follows from the (sharp) Poincar\'e inequality and the standard variational technique.  Note also the following simple lemmas:
\begin{lemma}
Suppose $g\geq 0$ is continuous on $[0,L]$ and $G$ is the unique solution to \eqref{ellipticproblem} on $[0,L]$. Then, $G''$ is continuous, $G\leq 0$ on $[0,L]$, and consequently $G''\geq g$ on $[0,L].$
\end{lemma}
\begin{proof}Continuity follows from the fact that $G\in H^2([0,L])$ and Sobolev embedding. 
Now suppose $G>0$ somewhere on $[0,L]$, then by passing to a slightly smaller interval, we can assume that $G\geq 0$ on $[0,L]$. Next, we multiply \eqref{ellipticproblem} by $G$ and integrate to see:
$$\int_0^L 4G^2 d\theta -\int_0^L G'^2 d\theta \geq 0,$$ which contradicts the sharp Poincar\'e inequality: $$\int_0^L G^2 d\theta \leq \frac{L^2}{\pi^2}\int_0^L G'^2 d\theta,$$ since $L<{\pi}/{2}$. 
\end{proof}
\begin{lemma}
If $g,g'\geq 0$, then $-G'(0),G'(L), G'(0)+G'(L)\geq 0$. 
\end{lemma}
\begin{proof}
Note that since $G\leq 0$ and $G(0)=G(L)=0$ we must have $-G'(0),G'(L)\geq 0$. 
Multiplying \eqref{ellipticproblem} by $\sin(2\theta)$ and $\cos(2\theta)$ and integrating gives us the following two identities:
$$\sin(2L) G'(L) = \int_0^L g(\theta)\sin(2\theta)d\theta,$$ and
$$\cos(2L)G'(L)-G'(0)=\int_0^L g(\theta) \cos(2\theta)d\theta,$$
which imply: 
$$G'(L) = \frac{1}{\sin(2L)}\int_0^L g(\theta)\sin(2\theta)d\theta,$$
and
$$G'(0)=\frac{\cos(2L)}{\sin(2L)}\int_0^{L}g(\theta)\sin(2\theta)d\theta-\int_0^L g(\theta)\cos(2\theta)d\theta.$$
Thus, $$G'(L)+G'(0)= \int_0^L g(\theta) \Big(\sin(2\theta)\frac{1+\cos(2L)}{\sin(2L)}-\cos(2\theta)\Big)d\theta.$$
Now define $K(\theta):=\sin(2\theta)\frac{1+\cos(2L)}{\sin(2L)}-\cos(2\theta),$ and note that $K(0)=-1$ and $K(L)=1$. In fact, note the following identity:
$$K(L-\theta)=-K(\theta).$$ Indeed, \begin{equation*}
\begin{split}
K(L-\theta) &= \sin(2(L-\theta))\frac{1+\cos(2L)}{\sin(2L)}-\cos(2(L-\theta)) \\
&=(\sin(2L)\cos(2\theta)-\cos(2L)\sin(2\theta))\frac{1+\cos(2L)}{\sin(2L)}-\cos(2L)\cos(2\theta)-\sin(2L)\sin(2\theta) \\
&= \cos(2\theta)-\sin(2\theta)\frac{\cos(2L)+1}{\sin(2L)}=-K(\theta).
\end{split}
\end{equation*} 
Note also that $K(\theta)\leq 0$ on  $[0,L/2]$ and $K(\theta)\geq 0$ on $[L/2,L]$
Hence, \begin{equation*}
\begin{split}
G'(L)+G'(0)=\int_0^L g(\theta) K(\theta) d\theta &=\int_0^{L/2} g(\theta)K(\theta)d\theta +\int_{L/2}^L g(\theta)K(\theta)d\theta \\
&=\int_{L/2}^{L} (g(\theta)- g(L-\theta)) K(\theta)\geq 0,
\end{split}
\end{equation*}  since $g'\geq 0$. Now we are done. 
\end{proof}
\begin{lemma}
If $g,g'\geq 0$ and $L<{\pi}/{2}$ then there exists $c=c(L)>0$ so that 
$$G'(L)\geq c \int_0^{L} g(\theta)d\theta.$$
\end{lemma}
\begin{proof}
Recall from the proof of the last lemma that $$G'(L)=\frac{1}{\sin(2L)}\int_0^L g(\theta)\sin(2\theta)d\theta\geq c(L)\int_0^L g(\theta)d\theta.$$ This finishes the proof.
\end{proof} 
\section{Blow-up for solutions with vorticity of compact support}\label{sec:compact_blowup}
In this section, we show that solutions to the $2D$ Boussinesq system \eqref{B1} -- \eqref{B2} with vorticity and density of compact support could blow up in finite time. This is done by taking initial data which is the sum of a radially homogeneous part and a remainder term which is smooth. The following result shows that if the radially homogeneous part blows up, then it implies blow up for the $2D$ solution. (The analogous theorem for the SQG equation has appeared in \cite{EJSI}.)
\begin{theorem}\label{thm:1Dplus2D}
Consider initial data $(\omega_0,\nabla^\perp\rho_0) \in \mathring{C}^{\alpha}(\overline{\Omega})$ such that $\omega_0$ and $\partial_{x_2}\rho_0$ are odd and $\partial_{x_1}\rho_0$ is even. Further assume that there is a decomposition \begin{equation}\label{eq:decomp}
\begin{split}
\omega_0 &= \omega_0^{1D} + \tilde{\omega}_0, \qquad
\rho_0 = \rho_0^{1D} + \tilde{\rho}_0,
\end{split}
\end{equation} and $\omega_0^{1D}, \nabla^\perp\rho_0^{1D}$ have the same set of symmetries as $\omega_0$ and $\nabla^\perp\rho_0$. In the decomposition we assume that for some $0 < \alpha <1$, $\omega_0^{1D}$ and $\rho_0^{1D}$ are radially homogeneous and satisfies \begin{equation*}
\begin{split}
&\omega_0^{1D}(r,\theta) = g_0(\theta),\qquad g_0 \in C^{1,\alpha}([-\pi/4,\pi/4]),\\
&\rho_0^{1D}(r,\theta) = rP_0(\theta),\qquad P_0 \in C^{2,\alpha}([-\pi/4,\pi/4])
\end{split}
\end{equation*} and the remainder satisfies \begin{equation*}
\begin{split}
\tilde{\omega}_0,\quad \nabla^\perp\tilde{\rho}_0 \in C^{\alpha}(\overline{\Omega}) \qquad \mbox{and} \qquad \tilde{\omega}_0(0) = 0,\quad \nabla^\perp\tilde{\rho}_0(0) = (0,0)^T. 
\end{split}
\end{equation*}
\begin{enumerate}
\item (Local well-posedness for the decomposition) There exists some $T > 0$ depending only on the norms of the initial data $\omega_0^{1D}, \tilde{\omega}_0, \rho_0^{1D}$, and $\tilde{\rho}_0$ that the unique local-in-time solution  $(\omega,\nabla^\perp\rho) \in L^\infty([0,T];\mathring{C}^{\alpha}(\overline{\Omega}))$ with $(\omega_0,\nabla^\perp\rho_0)$ has the form \begin{equation*}
\begin{split}
\omega = \omega^{1D} + \tilde{\omega}, \qquad \rho = \rho^{1D} + \tilde{\rho}.
\end{split}
\end{equation*} In this decomposition, $\omega^{1D}(r,\theta) = g(\theta)$ and $\rho^{1D}(r,\theta) = rP(\theta)$ where $(g,P)$ is the unique local-in-time solution belonging to $C^0([0,T];C^{1}\times C^{2}[-\pi/4,\pi/4]) $ of the $1D$ system \eqref{BSI1} -- \eqref{BSI2} with initial data $(g_0,P_0)$. Moreover, the remainder part retains the usual H\"older regularity: $\tilde{\omega},\nabla^\perp\tilde{\rho} \in C^0([0,T];C^{\alpha}(\overline{\Omega}))$. The lifespan $T > 0$ of this decomposition can be extended as long as both of $(\omega,\rho)$ and $(g,P)$ do  not blow up for their respective systems. 
\item ($1D$ blow-up implies $2D$ blow-up) Assume that the solution of the $1D$ system $(g,P)$ blows up at some finite time $T^* $. Then, the solution $(\omega,\rho)$ to the $2D$ system also blows up at some time $0 < T' \le T^*$, that is, \begin{equation*}
\begin{split}
\limsup_{t \nearrow T'} \left( \V \omega_t\V_{\mathring{C}^{\alpha}(\overline{\Omega})} + \V \nabla^\perp\rho_t\V_{\mathring{C}^{\alpha}(\overline{\Omega})} \right) = + \infty. 
\end{split}
\end{equation*}
\end{enumerate}
\end{theorem}
Given the conditional blow-up result, the $1D$ blow-up result from Section \ref{sec:1D} immediately finishes the proof of Theorem \ref{MainThm3}.
\begin{proof}[Proof of Theorem \ref{MainThm3}]
We simply take initial data of the form $\omega_0 = \omega_0^{1D} + \tilde{\omega}_0$ and $\rho_0 = \rho_0^{1D} + \tilde{\rho}_0$, where $(\omega^{1D}_0,\rho^{1D}_0)$ is radially homogeneous and the corresponding solution of the $1D$ system blows up in finite time. Here, $\tilde{\omega}$ and $\tilde{\rho}$ can be chosen to be smooth and make $\omega$ and $\nabla^\perp\rho$ compactly supported. (For instance, $\omega_0,\nabla^\perp\rho_0$ can belong to $\mathring{C}^\infty$ and $\tilde{\omega}_0,\nabla^\perp\tilde{\rho}_0$ to $C^\infty$.) Since the vorticity is odd, the corresponding velocity decays as $|x|^{-2}$ as $|x| \rightarrow +\infty$ and in particular, it has finite energy. 
\end{proof}
\begin{proof}[Proof of local well-posedness for the decomposition]
We first pick $T_1 > 0$ such that the following holds: \begin{enumerate}
\item There is a unique solution $(g,P)$ to the $1D$ system on $[0,T_1]$, with $g, \partial_{\theta} P \in C^{1,\alpha}[-\pi/4,\pi/4]$.
\item There is a unique solution $(\omega,\rho)$ to the $2D$ system on $[0,T_1]$ with $\omega,\nabla^\perp\rho \in \mathring{C}^\alpha(\overline{D})$. 
\end{enumerate}
From now on, we shall work on the time interval $[0,T_1]$. We begin with defining \begin{equation*}
\begin{split}
\tilde{\omega} = \omega - g(\theta),\qquad \tilde{\rho} = \rho - rP(\theta). 
\end{split}
\end{equation*} It is straightforward to check that, $\tilde{\omega}$ and $\tilde{\rho}$ respectively satisfy the following equations: \begin{equation}\label{eq:omega_remainder}
\begin{split}
\pr_t\tilde{\omega} + (u^{1D} + \tilde{u})\cdot\nabla\tilde{\omega} + \tilde{u}\cdot\nabla\omega^{1D} = \partial_{x_2}\tilde{\rho},
\end{split}
\end{equation} and \begin{equation}\label{eq:rho_remainder}
\begin{split}
\pr_t\tilde{\rho} + (u^{1D} + \tilde{u})\cdot\nabla\tilde{\rho} + \tilde{u}\cdot\nabla \rho^{1D} = 0.
\end{split}
\end{equation} We restrict ourselves to obtaining appropriate a priori $C^\alpha$-estimates, for solutions of the system \eqref{eq:omega_remainder} -- \eqref{eq:rho_remainder}. Once it is done, it is straightforward to show that there actually exists $C^\alpha$ solution to the system for some time, and that this solution indeed coincides with the difference $\omega - g(\theta)$ and $\rho - rP(\theta)$.   We first observe an $L^\infty$-bound for $\tilde{\omega}$:\begin{equation*}
\begin{split} 
\frac{d}{dt} \V \tilde{\omega}\V_{L^\infty} &\le \V \nabla^\perp\tilde{\rho}\V_{L^\infty} + \V \tilde{u}\cdot\nabla g\V_{L^\infty}  \\
&\le \V \nabla^\perp\tilde{\rho}\V_{L^\infty}  + \left| \frac{\tilde{u}(x)}{|x|}\right|_{L^\infty} \cdot \V g \V_{\mathrm{Lip}} \le \V \nabla^\perp\tilde{\rho}\V_{L^\infty}  + \V \tilde{\omega}\V_{L^\infty}\V g \V_{\mathrm{Lip}}. 
\end{split}
\end{equation*}  Then in turn, \eqref{eq:rho_remainder} shows that $\tilde{\rho}(0) = 0$, and hence $|\tilde{\rho}(x)| \lesssim |x|$ as well. Next, by differentiating \eqref{eq:rho_remainder}, we obtain the evolution equation for $\nabla^\perp\tilde{\rho}$: \begin{equation}\label{eq:nablarho_remainder}
\begin{split}
\pr_t \nabla^\perp\tilde{\rho} + (u^{1D} + \tilde{u})\cdot\nabla (\nabla^\perp\tilde{\rho}) + \tilde{u}\cdot\nabla (\nabla^\perp \rho^{1D}) = \nabla u\nabla^\perp\tilde{\rho} + \nabla\tilde{u}\nabla^\perp\rho^{1D}. 
\end{split}
\end{equation} Then, \begin{equation*}
\begin{split}
\frac{d}{dt} \V \nabla^\perp\tilde{\rho} \V_{L^\infty} \le \left| \frac{\tilde{u}(x)}{|x|} \right|_{L^\infty} \cdot \V P \V_{\mathrm{Lip}} + \V \nabla u \V_{L^\infty} \left( \V \nabla^\perp\tilde{\rho} \V_{L^\infty} + \V P \V_{\mathrm{Lip}} \right). 
\end{split}
\end{equation*} 
At this point, we note that if $\tilde{\omega}$ is a bounded function on $\Omega$ and satisfies $|\tilde{\omega}(x)| \le C |x|^\alpha$, the corresponding velocity satisfies \begin{equation*}
\begin{split}
|\tilde{u}(x)| \lesssim |x|^{1+\alpha}. 
\end{split}
\end{equation*} To see this, we write \begin{equation*}
\begin{split}
\tilde{u}(x) = \frac{1}{2\pi} \int_{|x-y| \le 2|x|} \frac{(x-y)^\perp}{|x-y|^2}\tilde{\omega}(y)dy + \frac{1}{2\pi} \int_{|x-y| > 2|x|} \frac{(x-y)^\perp}{|x-y|^2}\tilde{\omega}(y)dy ,
\end{split}
\end{equation*} where $\tilde{\omega}$ have been extended to a bounded and 4-fold symmetric function on $\mathbb{R}^2$. Then, in the region $|x-y| \le 2|x|$, $|y| \le C|x|$, so that a direct integration gives \begin{equation*}
\begin{split}
\left|  \int_{|x-y| \le 2|x|} \frac{(x-y)^\perp}{|x-y|^2}\tilde{\omega}(y)dy  \right| \le C|x|^{1+\alpha}. 
\end{split}
\end{equation*} On the other hand, for the other region, one can symmetrize the kernel to obtain decay like $1/|y|^4$, for $x$ fixed (see \cite{E1} ,\cite{EJSI}), and then directly integrate using $|\tilde{\omega}(y)|\lesssim |y|^\alpha$ to again get a contribution of the form $C|x|^{1+\alpha}$. 
In particular, this forces $\nabla \tilde{u}(0) = 0$, and we also have $|\nabla\tilde{u}(x)|\lesssim |x|^\alpha$. Plugging in this information into \eqref{eq:nablarho_remainder}, it follows that $\nabla^\perp\tilde{\rho}(0) = 0$ as well. Hence, we have shown that \begin{equation}\label{eq:Holdervanishing}
\begin{split}
\frac{|\tilde{u}(x)|}{|x|},|\nabla \tilde{u}(x)| \le \V \nabla\tilde{u}\V_{C^{\alpha}}|x|^{\alpha},\qquad \mbox{and}\qquad \frac{|\tilde{\rho}(x)|}{|x|}, |\nabla^\perp\tilde{\rho}(x)| \le \V \nabla^\perp\tilde{\rho}\V_{C^\alpha}|x|^{\alpha}. 
\end{split}
\end{equation} We may therefore apply Lemma \ref{lem:Holder_product} to gradients of $\tilde{u}$ and $\tilde{\rho}$, from now on. 
We proceed to obtaining $C^\alpha$-bounds for $\tilde{\omega}$. Composing with the flow generated by $u = u^{1D} + \tilde{u}$, \begin{equation*}
\begin{split}
\pr_t \tilde{\omega}\circ X = \partial_{x_2}\tilde{\rho}\circ X - (\tilde{u}\cdot\nabla\omega^{1D})\circ X. 
\end{split}
\end{equation*} The flow map is Lipschitz in space, and note that \begin{equation*}
\begin{split}
\tilde{u}\cdot\nabla\omega^{1D} = g'(\theta) \frac{ \tilde{u}(x)\cdot x^\perp}{|x|^2}.
\end{split}
\end{equation*} From Lemma \ref{lem:Holder_product}, it follows that $\V \tilde{u}\cdot\nabla\omega^{1D} \V_{C^\alpha} \le C \V g\V_{C^{1,\alpha}} \V\nabla\tilde{u}\V_{C^\alpha}$. Hence \begin{equation*}
\begin{split}
\partial_{x_2}\tilde{\rho}\circ X - (\tilde{u}\cdot\nabla\omega^{1D})\circ X
\end{split}
\end{equation*} is bounded in $C^\alpha$, and arguing along the lines of the proof of Theorem \ref{thm:LWP2D}, one easily obtains an a priori $C^\alpha$ bound for $\tilde{\omega}$. The argument for $\nabla^\perp\tilde{\rho}$ is strictly similar: it suffices to obtain a $C^\alpha$-bound for \begin{equation*}
\begin{split}
\nabla u\nabla^\perp\tilde{\rho} + \nabla\tilde{u}\nabla^\perp\rho^{1D} - \tilde{u}\cdot\nabla(\nabla^\perp\rho^{1D}). 
\end{split}
\end{equation*} Each term can be shown to be bounded in $C^\alpha$ with a simple application of Lemma \ref{lem:Holder_product}. 
Hence, we have shown existence of appropriate  $C^\alpha$ a priori inequalities for $\tilde{\omega}$ and $\nabla^\perp\tilde{\rho}$. The $C^\alpha$-norms cannot blow up for some time interval $[0,T]$, with $0 < T \le T_1$. 
Lastly, we argue that the decomposition is valid as long as $(\omega,\rho)$ and $(g,P)$ stay regular for the $2D$ and the $1D$ systems, respectively. Let $T^*$ be a potential blow-up time for the system \eqref{eq:omega_remainder} -- \eqref{eq:nablarho_remainder} in $C^{\alpha}(\overline{\Omega})$. If $(\omega,\rho)$ and $(g,P)$ stay regular up to time $T^*$, it means in particular that $\V \nabla u\V_{L^\infty} + \V \nabla^\perp\rho\V_{L^\infty} + \V g\V_{\mathrm{Lip}} + \V P \V_{\mathrm{Lip}}$ remains bounded uniformly in the time interval $[0,T^*]$. In particular $\V \nabla\tilde{u}\V_{L^\infty} + \V \nabla^\perp\tilde{\rho}\V_{L^\infty}$ stays bounded as well, but this controls the blow-up for the system  \eqref{eq:omega_remainder} -- \eqref{eq:nablarho_remainder} in $C^{\alpha}$ (see for instance the arguments given in the proof of Theorem \ref{thm:LWP2D}). This finishes the proof. 
\end{proof}
\begin{proof}[Proof of $2D$ blow-up from $1D$ blow-up]
For the sake of contradiction, assume that the solution $(\omega,\nabla^\perp\rho)$ to the $2D$ system stays in $\mathring{C}^{\alpha}$ up to the time moment $T^*$. Then, in particular we have that \begin{equation}\label{eq:blowup_implications}
\begin{split}
\sup_{t \in [0,T^*]} \left( \V \nabla u_t \V_{L^\infty} + \V \nabla^\perp\rho_t \V_{L^\infty}\right) \le C
\end{split}
\end{equation} for some constant $C > 0$, as well as \begin{equation}\label{eq:blowup_implications2}
\begin{split}
\limsup_{t \nearrow T^*} \left(   \V g_t\V_{\mathrm{Lip}} + \V P_t\V_{\mathrm{Lip}}   \right) \rightarrow + \infty. 
\end{split}
\end{equation} 
Now fix some $t < T^*$ and taking $\limsup$ in space, \begin{equation*}
\begin{split}
C \ge \limsup_{|x| \rightarrow 0} |\nabla^\perp\rho_t(x)| = \limsup_{|x| \rightarrow 0} |\nabla^\perp \rho^{1D}_t(x)| \ge c\V P\V_{\mathrm{Lip}},
\end{split}
\end{equation*} simply because we have $|\nabla^\perp\tilde{\rho}(x)|\lesssim |x|^{\alpha}$ as long as $t < T^*$. Similarly, we deduce that \begin{equation*}
\begin{split}
C \ge c\V g\V_{\mathrm{Lip}}.
\end{split}
\end{equation*} The constants $c, C > 0$ are independent of $t$. Taking the limsup for $t \nearrow T^*$ gives a contradiction. 
\end{proof}
\section{Blow-up for $C^\infty$ solutions with bounded vorticity}\label{sec:smooth_blowup}
In this section, we give a sketch of the proof of Theorem \ref{MainThm4}. While there have been many proofs of blow-up for infinite energy $C^\infty$ solutions to the incompressible Euler equations (even in $2D$!) and the Boussinesq system, none of them seem to be have bounded vorticity. In fact, for the stagnation-point ansatz solutions constructed in \cite{Con}, \cite{CISY}, and \cite{SaWu}, the vorticity and all of its derivatives are actually growing linearly at spatial infinity. Those solutions do not approximate compactly supported data in $2D$ no matter how large the support is.  The $C^\infty$ blowing-up solutions we construct do approximate compactly supported data and, thus, have some physical significance. In the previous section, we showed that if we "cut"  scale-invariant data at infinity -- so that they are compactly supported -- then the resulting compactly supported solution must blow-up at a time $t\leq T^*$ where $T^*$ is the blow-up time for the scale-invariant solution. We did this (via a contradiction argument) by studying $$\limsup_{|x|\rightarrow 0} |\omega(x)|.$$ Here, we will do the same except that we will cut out the area around $(0,0)$ which will automatically make the data and solution $C^{\infty}$ and then we will show that if a solution exists up to $T^*$, then $$\V \omega \V_{L^\infty}\geq \limsup_{|x|\rightarrow\infty}|\omega(x,t)|\rightarrow\infty$$ as $t\rightarrow T^*$. Since the proof of Theorem \ref{MainThm4} is so similar to the proof of \ref{MainThm3}, we will only give a sketch. 
\begin{proof}[Proof of Theorem \ref{MainThm4}]
Just as in the proof of Theorem \ref{MainThm3}, we proceed by taking initial data 
$$\omega_0 = \omega_0^{1D}+\tilde\omega_0,\qquad \rho_0 = \rho_0^{1D}+\tilde\rho_0$$
with $$\omega_0^{1D}=g_0(\theta),\qquad \rho_0^{1D}=r P_0(\theta)$$ and $\omega_0, \rho_0\in C^\infty(\Omega)$ and supported away from the origin, while $\text{supp}(\tilde\omega_0), \text{supp}(\tilde\rho_0)\subset B_1(0).$ Towards a contradiction, let us assume that the local solution\footnote{It is not difficult to show that the local well-posedness theorem, Theorem \ref{thm:LWP2D}, can be proven for $C^{k,\alpha}(\Omega)$ instead of $\mathring{C}^{k,\alpha}$ solutions, for any $k \ge 1$, as long as the data is supported away from the origin.} emanating from $(\omega_0, \rho_0)$ is global. We wish to prove that if $(\omega^{1D},\rho^{1D}):=(g(t,\theta), P(t,\theta))$ is the unique solution to  \eqref{BSI1} -- \eqref{BSI2} starting from $(g_0,P_0)$, then 
\begin{equation}\label{keyestimate1}\lim_{|x|\rightarrow \infty}|\omega(t)-\omega^{1D}(t)|=0\end{equation}
for all $t\in [0,T^*)$ where $T^*$ is such that $\omega$ and $g$ remain bounded for all $0\leq t<T^*.$ In the following, $T^*$ will be the blow-up time of $(g,P)$. Let us define $\tilde\omega=\omega-\omega^{1D}$ and $\tilde\rho=\rho-\rho^{1D}.$ As before, $\tilde\omega$ and $\tilde\rho$ satisfy the following system:   
\begin{equation}\label{eq:omega_remainder2}
\begin{split}
\pr_t\tilde{\omega} + u\cdot\nabla\tilde{\omega} + \tilde{u}\cdot\nabla\omega^{1D} = \partial_{x_2}\tilde{\rho},
\end{split}
\end{equation} and \begin{equation}\label{eq:rho_remainder2}
\begin{split}
\pr_t\tilde{\rho} + u\cdot\nabla\tilde{\rho} + \tilde{u}\cdot\nabla \rho^{1D} = 0.
\end{split}
\end{equation} 
At time $t=0$, $\tilde\omega$ and $\tilde\rho$ are compactly supported. We will show that on $[0,T^*)$, \begin{equation} \label{decayestimate} |\nabla\tilde\rho(x)|+|\tilde\omega(x)|\lesssim \frac{1}{|x|} \end{equation} for $|x|>1$ which implies \eqref{keyestimate1}. We remark that the estimate is allowed to degenerate as $t\rightarrow T^*$. Now note that 
$$\partial_t (|x|\tilde\omega)+\tilde\omega u\cdot\nabla|x| + u\cdot\nabla (|x|\tilde\omega)+ \frac{\tilde u\cdot x^\perp}{|x|}=|x|\partial_{x_2} \tilde\rho.$$
It is easy to see that, $$\frac{d}{dt}\V |x|\tilde\omega\V_{L^\infty} \leq  \V |x|\tilde\omega\V_{L^\infty} \V \nabla u\V_{L^\infty}+ \left\V \frac{\tilde{u}\cdot x^\perp}{|x|}\right\V_{L^\infty}+ \V |x|\nabla\rho\V_{L^\infty} $$ and that a similar estimate holds for $|x|\nabla\tilde\rho.$
Now, the key observation is: 
\begin{equation} \label{keyestimate2} \V \tilde u\V_{L^\infty(\Omega)} \leq \V |x|\tilde\omega\V_{L^\infty(\Omega)}. \end{equation} Indeed, observe that once we symmetrize the Biot-Savart kernel in the proper way, and make a change of variables, we have 
$$\tilde u(x) = |x|\int_{\mathbb{R}^2} \tilde{K}(\frac{x}{|x|},y)\tilde\omega(|x|y)dy$$ which directly implies \eqref{keyestimate2}.\footnote{Here, $\tilde{K}(x,y) := (K(x,y) + K(x,y^\perp) + K(x,-y) + K(x,-y^\perp))/4$ with $K(x,y) := (x-y)^\perp/2\pi|x-y|^2$ being the usual Biot-Savart kernel, and $\tilde{\omega}$ have been extended to be a 4-fold symmetric function on $\mathbb{R}^2$.} 
Similarly, $$\frac{d}{dt}\V |x|\nabla\tilde\rho\V_{L^\infty} \lesssim  \left(\V |x|\tilde\omega\V_{L^\infty}+\V |x|\nabla\tilde\rho\V_{L^\infty} \right) \V\nabla u\V_{L^\infty}+\V P\V_{W^{2,\infty}} \V|x|\nabla\tilde\omega\V_{L^\infty}.$$
This establishes\footnote{Strictly speaking, we should prove estimates on $\V |x|\exp(-\epsilon |x|^2)\tilde\omega\V_{L^\infty}$ which are independent of $\epsilon>0$ and then pass to the limit to get \eqref{decayestimate}. This can be done in the same way.} \eqref{decayestimate} which implies: $$\lim_{|x|\rightarrow \infty} |\omega(x)- \omega^{1D}|=|\tilde\omega(x)|=0.$$ 
However, $$\limsup_{|x|\rightarrow\infty} |\omega^{1D}|=\V \omega^{1D}\V_{L^\infty}\rightarrow\infty$$ as $t\rightarrow T^*$. This implies that:
$$\V\omega\V_{L^\infty}\geq \limsup_{|x|\rightarrow\infty}|\omega(x)|\geq\limsup_{x\rightarrow\infty}|\omega^{1D}|-\lim_{x\rightarrow\infty}|\omega-\omega^{1D}|=\limsup_{x\rightarrow\infty}|\omega^{1D}|.$$
As a consequence, $\V\omega\V_{L^\infty}\rightarrow \infty$ as $t\rightarrow T^*$ which is a contradiction. Hence, the unique local solution with data $\nabla\rho_0, \omega_0\in C^{\infty}\cap L^\infty(\Omega)$ loses regularity in finite time. 
 \end{proof}

\section{Blow up for finite-energy H\"older continuous solutions in the case of acute corners}\label{sec:acutecase}

Note that the results of the previous sections give local solutions which are $C^\infty$ away from the origin in the domain $\{(r,\theta) : -\beta\pi< \theta < \beta\pi \}$ for any $0<\beta<1/2$ and become singular in finite time. In particular, with those results, we can get arbitrarily close to the case of the half-space which is almost identical to the scenario in the numerical work \cite{HouLuo}. In this section, we prove a stronger result when the corner is acute -- that is, we take domains \begin{equation*}
\begin{split}
\Omega_\beta = \{ (r,\theta) : 0< \theta < \beta\pi   \}
\end{split}
\end{equation*} with some $\beta < 1/2$, and then we can show finite time singularity formation for compactly supported initial data $(\omega_0,\nabla\rho_0)$ which is uniformly $C^\alpha$ up to the corner. It is not clear what bearing this result has on the half-space case, but it is interesting to point out nonetheless.  

\begin{theorem}\label{AcuteCase}
Let $\beta < 1/2$ and for some integer $k \ge 0$, $0 < \alpha < 1/\beta - k - 2$. If $\omega_0,\nabla\rho_0\in C^{k,\alpha}(\overline{\Omega}_\beta)$ are such that $\omega(0)>0$ and $\partial_{x_2}\rho_0(0)>0$, then the local unique $C^\infty$ solution to \eqref{B1} -- \eqref{B2} becomes singular in finite time. 
\end{theorem}


\begin{proof}
The essential observation is that when $\beta<1/2$, $\nabla u(0)$ depends \emph{only} on $\omega(0)$. This allows one to write an ODE system on $\nabla\rho(0)$ and $\omega(0)$ and the blow-up can be seen directly. Notice that if $\omega\in C^{k,\alpha}(\overline{\Omega}_\beta)$, and $\Psi$ is the unique solution  satisfying $\Delta\Psi=\omega$ on $\Omega_\beta$ and $\Psi=0$ on $\partial\Omega_\beta$, then since $k + \alpha < 1/\beta -2$, we have $\Psi \in C^{k+2,\alpha}(\overline{\Omega}_\beta)$. In particular, Taylor's theorem gives the expansion $$\Psi(x_1,x_2) = \frac{\omega(0)}{2(1-\beta^2)}(x_1^2-\beta^2 x_2^2) +o(|x|^2).$$ Now evaluating the system at the origin, we have  $$\partial_t \omega(0)=\partial_{x_2}\rho(0),$$ $$\partial_t \partial_{x_1}\rho(0)=-\partial_{x_1} u(0)\cdot\nabla\rho(0),$$ and
$$\partial_t \partial_{x_2}\rho(0)=-\partial_{x_2}u(0)\cdot\nabla\rho(0).$$
Now notice: $\partial_{x_1}u_1(0)=\partial_{x_2} u_2(0)=0$, $\partial_{x_1} u_2(0)=\frac{\omega(0)}{(1-\beta^2)}$, and $\partial_{x_2} u_1(0)=\frac{\beta^2}{(1-\beta^2)}\omega(0)$, using the Taylor expansion of $\Psi$ and the fact that $\nabla^\perp \Psi=\omega$. 
Hence, the above set of equations can be re-written as 
$$\partial_t \omega(0)=\partial_{x_2}\rho(0)$$ $$\partial_t \partial_{x_1}\rho(0)=-\frac{\omega(0)}{1-\beta^2}\partial_{x_2}\rho(0)$$
$$\partial_t \partial_{x_2}\rho(0)=-\frac{\beta^2}{(1-\beta^2)}\omega(0)\partial_{x_1} \rho(0).$$ To clarify things, we introduce $A=\omega(0),$ $B=\partial_{x_1}\rho(0)$, and $C=\partial_{x_2}\rho(0)$ to obtain the system 
$$A'=C,$$
$$B'=-\frac{1}{1-\beta^2}AC,$$
$$C'=-\frac{\beta^2}{1-\beta^2}AB,$$ and it is easy to see that if $A_0,C_0\geq 0$ and $B_0\leq 0$, then we have singularity formation in finite time so long as $C_0>0$.
Alternatively, using the cut-off argument in the previous sections, it suffices to exhibit a velocity field and density profile which are homogeneous of degree 1, which satisfy the boundary condition ($u\cdot n=0$ on the boundary), which are $C^{k+1,\alpha}(\overline{\Omega}_\beta)$, and which become singular in finite time. 
\end{proof}

\begin{remark}
	In an interesting work of Chae, Constantin, and Wu \cite{CCW2}, the authos consider solutions to the $3D$ Euler and SQG equations satisfying various symmetry assumptions and explore potential finite-time singularity formation on the lower dimensional set left fixed by the symmetry upon consideration, by writing down the simplified system on that invariant set similar to the one appearing in the proof of Theorem \ref{AcuteCase}. Our result show that in our specific setting of a sector, the reduced system for the Boussinesq system can be closed by itself. A similar phenomenon occurs also in the case of the $3D$ Euler equations and this will be treated in our forthcoming work. 
\end{remark}

\section{Conclusion}

In \cite{EJSI} we introduced a method for establishing singularity formation for inviscid incompressible fluid equations using scale-invariant solutions. The method essentially relies on the study of scale-invariant solutions which (generally) satisfy a simpler evolution equation. One of the difficulties one faces in trying to carry this out is that the basic fluid equations are generally ill-posed in borderline regularity classes (see \cite{EM1} and \cite{BL2}) and scale-invariant solutions are exactly exactly at the borderline regularity. This led us to define a scale of spaces, which we call $\mathring{C}^\alpha$ which are small enough to overcome ill-posedness but large enough to accommodate scale-invariant solutions. In this work, we applied the method to the Boussinesq system on a class of domains. It is likely that this method has much wider applicability.  One direction for further inquiry would be to examine how flexible the cut-off argument above is--in particular, is it possible to establish singularity formation when one cuts a scale-invariant solution \emph{both} at $0$ and at $\infty$? Another direction of interest would be to extend this analysis to the half-space case $\mathbb{R}^2_+$ by possibly combining the ideas here with the ideas used in \cite{KS}.   

\section*{Acknowledgments}
We are grateful to Fanghua Lin and Yuan Cai for pointing out a mistake in the proof of Theorem \ref{thm:1Dblowup} in an earlier version of the manuscript. We also thank Vladimir \v{S}ver\'ak for many comments which greatly improved the exposition. 

\bibliographystyle{plain}
\bibliography{Boussinesq_final}
\end{document}